\documentclass[11pt,a4paper]{article}
\usepackage{siunitx}
\usepackage{graphicx}
\usepackage{tabularx,booktabs,multirow,siunitx}
\usepackage{subcaption} 
\usepackage{caption}
\usepackage{amsmath}
\usepackage{algorithm,algorithmic}
\usepackage{amsthm}
\usepackage{mathrsfs}
\usepackage{cases}
\usepackage{tikz}
\usepackage{amsfonts}
\usepackage{longfbox}
\usepackage{tikz-cd}
\usetikzlibrary{arrows,shapes}
\usetikzlibrary{matrix}
\usepackage{float}
\usepackage{algorithm}
\usepackage{amssymb}
\usepackage{tcolorbox}
\usepackage{rotating}
\usepackage[margin=1in]{geometry}
\makeatletter
\DeclareFontShape{OT1}{cmtt}{m}{scit}{<->ssub*cmtt/m/n}{}
\makeatother
\newcolumntype{L}[1]{>{\raggedright\let\newline\\\arraybackslash\hspace{0pt}}m{#1}}
\newcolumntype{C}[1]{>{\centering\let\newline\\\arraybackslash\hspace{0pt}}m{#1}}
\newcolumntype{R}[1]{>{\raggedleft\let\newline\\\arraybackslash\hspace{0pt}}m{#1}}

\newtheorem{Theorem}{Theorem}[section]

\newtheorem{Remark}[Theorem]{Remark}
\newtheorem{Lemma}[Theorem]{Lemma}

\newtheorem{Definition}[Theorem]{Definition}
\newtheorem{Example}[Theorem]{Example}

\usepackage{hyperref}
\hypersetup{colorlinks=true,urlcolor=blue}
\expandafter\let\expandafter\oldproof\csname\string\proof\endcsname
\let\oldendproof\endproof
\renewenvironment{proof}[1][\proofname]{
\oldproof[\ttfamily\scshape \bf #1.]
}{\oldendproof}
\def\ve{\varepsilon}

\def\dom{{\rm dom}\,}

\def\B{{\rm I\!B}}

\def\disp{\displaystyle}
\def\tto{\rightrightarrows}

\def\ra{\rangle}
\def\la{\langle}
\def\e{\varepsilon}
\def\epsilon{\varepsilon}

\def\ov{\overline{v}}

\def\ou{\overline{u}}

\def\dom{\mbox{\rm dom}\,}

\def\O{\Omega}

\def \N{{\rm I\!N}}
\def \R{{\rm I\!R}}
\def\th{\theta}

\numberwithin{equation}{section}

\begin{document}

\title{\bf Inexact DC Algorithms in Hilbert Spaces with Applications to PDE-Constrained Optimization}
\author{P. D. Khanh\footnote{Group of Analysis and Applied Mathematics, Department of Mathematics, Ho Chi Minh City University of Education, Ho Chi Minh City, Vietnam. E-mail: khanhpd@hcmue.edu.vn.} \quad V. V. H. Khoa\footnote{Department of Mathematics, Wayne State University, Detroit, Michigan, USA. E-mail: khoavu@wayne.edu. Research of this author was partly supported by the US National Science Foundation under grant DMS-2204519.}\quad B. S. Mordukhovich\footnote{Department of Mathematics and Institute for Artificial Intelligence and Data Science, Wayne State University, Detroit, Michigan, USA. E-mail: aa1086@wayne.edu. Research of this author was partly supported by the US National Science Foundation under grant DMS-2204519 and by the Australian Research Council under Discovery Project DP250101112.}\quad D. B. Tran\footnote{Department of Mathematics, Rowan University, Glassboro, New Jersey, USA. E-mail: trandb@rowan.edu.}\quad N. V. Vo\footnote{Department of Mathematics and Statistics, Oakland University, Rochester, Michigan, USA. E-mail: nghiavo@oakland.edu.}}
\maketitle
\vspace*{-0.2in}
\begin{quote}
{\small\noindent {\bf Abstract.} In this paper, we design and apply novel inexact adaptive algorithms to deal with minimizing difference-of-convex (DC) functions in Hilbert spaces. We first introduce I-ADCA, an inexact adaptive counterpart of the well-recognized DCA (difference-of-convex algorithm), that allows inexact subgradient evaluations and inexact solutions to convex subproblems while still guarantees global convergence to stationary points. Under a Polyak-\L ojasiewicz type property for DC objectives, we obtain explicit convergence rates for the proposed algorithm. Our main application addresses elliptic optimal control problems with control constraints and nonconvex $L^{1-2}$ sparsity-enhanced regularizers admitting a DC decomposition. Employing I-ADCA and appropriate versions of finite element discretization leads us to an efficient procedure for solving such problems with establishing its well-posedness and error bound estimates confirmed by numerical experiments.\\
\vspace*{0.03in}\noindent
{\bf Keywords.}\ inexact difference-of-convex-functions algorithm, Polyak-\L ojasiewicz conditions, PDE-constrained optimization, sparse controls, finite element discretization  
\\
{\bf Mathematics Subject Classification (2020).}\ 49J20, 49M41, 49N05, 65B99, 90C26}
\end{quote}

\section{Introduction}
\label{sec:intro}
DC (difference-of-convex) programming has emerged as a powerful framework for tackling a broad class of nonconvex optimization problems that retain a rich variational structure. Since the seminal work by Pham Dinh Tao and Le Thi Hoai An \cite{DCA97}, the {\em difference-of-convex algorithm} (DCA) and its variants have been successfully applied to numerous models in statistics, machine learning, and signal processing including feature selection, robust regression, support vector machines, and binary logistic regression; see, e.g., the surveys in \cite{30yearDCA,openDCA} for a comprehensive account of developments. In many of these applications, the DC structure arises naturally from modeling considerations such as robustness, sparsity, and nonconvex regularization while DCA provides a simple yet effective first-order scheme that alternates between convex subproblems and exhibits good practical performance even in large-scale settings. More recently, Nesterov-type inertial techniques initiated in \cite{Nesterov} have been incorporated into DCA to further accelerate its convergence. In particular, the accelerated DCA (ADCA) introduced in \cite{ijcai} (see also \cite{PL24MOR}) combines DCA with an extrapolation step and has demonstrated state-of-the-art performance for sparse binary logistic regression and other finite-dimensional DC models.

Despite this progress, various existing DC algorithms and their accelerated and/or inexact variants \cite{amp24,AFV18,fmss26,30yearDCA,openDCA,DCA97,PHLH22,ijcai,PL24MOR} are formulated and analyzed in finite-dimensional Euclidean spaces, with a few exceptions as in \cite{Auchmuty89,Yen24,Yen23,ZSYD24}. However, many modern applications in inverse problems, optimal control, and PDE-constrained optimization are inherently infinite-dimensional being posed naturally in Hilbert (sometimes in more general) spaces. Typical examples include distributed optimal control of semilinear elliptic and parabolic equations \cite{Antil,hinter02,Reyes,Troltzsch,Ulbrich11}, and variational problems with function-space sparsity constraints \cite{Stadler09,Wachsmuth09}. In such settings, a direct “discretize then optimize” approach may obscure the underlying DC structure at the continuous level and may lead to algorithms whose behavior depends in subtle ways on the discretization. This motivates the development of DC algorithms {\em directly in Hilbert spaces}, together with convergence theory and rate estimates that are robust with respect to discretization. To the best of our knowledge, a systematic development of {\em inexact DCA-type methods} to Hilbert spaces involving inexact computations in both subgradient evaluations and subproblem solving, has not yet been carried out.

A second major motivation for this work comes from {\em sparse optimal control of elliptic PDEs}. Consider, for instance, the classical distributed optimal control problem for a linear elliptic equation with control constraints (see, e.g., \cite[page~4]{Troltzsch}):
\begin{eqnarray*}
\underset{y \in H_0^1(\Omega),\, u \in L^2(\Omega)}{\operatorname{minimize}} && \frac{1}{2} \|y-y_d\|_{L^2(\Omega)}^2+\frac{\alpha}{2} \|u-u_d\|^2_{L^2(\Omega)} \\
\text {subject to } && -\Delta y=u \quad\quad \text { in } \Omega, \\
&& \beta_1 \leq u \leq \beta_2 \quad \text { a.e. in } \Omega,
\end{eqnarray*}
where the control $u$ enters as a source term, and where the aim is to track a desired state $y_d$. When a standard quadratic control cost is used, optimal controls are typically non-sparse and spread over the entire domain. To promote {\em sparsity} (and hence to reduce computation costs), Stadler \cite{Stadler09} augments the objective with a {\em nonsmooth term} $\beta\|u\|_{L^1 (\Omega)}$, which drives the optimal control to zero on subregions where further reduction of the cost is negligible; the threshold for what counts as negligible is set by $\beta>0$. In those areas, no actuation is applied. This mechanism is central to optimal placement of control devices, e.g., on piezoelectric plates associated with linear material laws. The reader is referred to \cite{CasasSIOPT12,HandbookNA,Wachsmuth09} and the monographs \cite{Ulbrich09,Ulbrich11,Reyes} for further developments and error bound analysis. The resulting $L^1$-regularized control problems have been extensively studied over the last decade, both analytically and numerically, and form a well-established benchmark in PDE-constrained optimization.

At the same time, there exists mounting evidence from finite-dimensional compressed sensing and imaging that the difference of the $\ell^1$ and $\ell^2$ norms, 
\begin{equation*}
\ell^{1-2}(x) := \|x\|_1 - \|x\|_2, \quad x\in \R^n,
\end{equation*}
can promote sparsity more effectively than the plain $\ell^1$ norm in highly coherent settings. In particular, it is shown in \cite{Lou} that $\ell^{1-2}$-based models yield sparser solutions and improved recovery performance in several compressed-sensing and image-denoising tasks. This suggests replacing the $L^1$ control cost in PDE-constrained optimization by an $L^{1-2}$-type functional as
\begin{equation*}
L^{1-2}(u):=\|u\|_{L^1 (\Omega)} - \|u\|_{L^2 (\Omega)},\quad u\in L^2 (\Omega),
\end{equation*}
with the goal of enhancing sparsity at the level of distributed controls. The resulting sparsity-enhancing regularizer is nonsmooth and nonconvex, but it admits a DC decomposition and hence fits into the DC programming framework. Quite recently, a certain inexact DCA with a sieving strategy (iDCA) has been proposed in \cite{ZSYD24} for discretized PDE-constrained optimal control problems with an $L^{1-2}$ control cost. 

The present paper aims to bridge these developments by formulating and analyzing inexact adaptive DC algorithms directly in Hilbert spaces and by applying them to {\em $L^{1-2}$–regularized elliptic optimal control} problems. We consider general DC programs of the form
\begin{equation*}
\min_{u\in \mathcal{H}} f(u) := g(u) - h(u) 
\end{equation*}
defined on an arbitrary Hilbert space $\mathcal{H}$, where $g$ and $h$ are proper, lower semicontinuous, and convex. Our first main contribution is to develop the {\em inexact adaptive DC algorithm} (I-ADCA) for solving such problems in Hilbert spaces. The proposed scheme incorporates a general form of Nesterov-type extrapolation step, and it allows for inexactness in two places: (i) $\varepsilon$–subgradients of $h$ computed at ``extrapolated'' points, and (ii) $\varepsilon$–approximate minimizers of the associated convex subproblems involving $g$. At each iteration, the inexactness levels are controlled online by a backtracking-type (adaptive) procedure. More precisely, I-ADCA adjusts the accuracy parameter at each iteration to ensure that the inexactness is commensurate with the actual progress of the algorithm that leads us to a family of estimates in the form
\begin{equation*}
\varepsilon_k \lesssim \|u^{k+1}-w^k\|^2
\end{equation*}
and hence to a genuine descent of the objective values. Building upon advanced tools of convex and variational analysis in Hilbert spaces \cite{Mordukhovich06,nam,Mordukhovich24,Zalinescu}, we prove that I-ADCA is well-defined and preserves fundamental convergence features of the exact accelerated DCA \cite{ijcai} developed in finite dimensions. Under additional assumptions, most notably, a {\em Polyak-\L ojasiewicz type} property for the objective $f$ and local Lipschitz continuity of either $\nabla g$ or $\nabla h$, we establish explicit convergence rates for the objective values generated by I-ADCA. To the best of our knowledge, this is the {\em first convergence-rate} result for a DC algorithm that simultaneously incorporates {\em acceleration, inexactness, and infinite dimensionality} of the space in question.

Yet another major contribution of this paper is applying the proposed scheme to {\em elliptic optimal control} problems with {\em $L^{1-2}$ sparsity-enhanced regularization}. Consider the problem
\begin{equation*}
\min _{y \in H_0^1(\Omega),\, u \in L^2(\Omega)}\left\{\frac{1}{2}\left\|y-y_d\right\|_{L^2(\Omega)}^2+\frac{\alpha}{2}\left\|u-u_d\right\|_{L^2(\Omega)}^2+\beta\left(\|u\|_{L^1(\Omega)}-\|u\|_{L^2(\Omega)}\right)\right\}
\end{equation*}
subject to a linear elliptic PDE and pointwise control bounds; see \eqref{eq:mainprob} for the precise formulation. We show that, in the Hilbert space setting, the associated cost functional can be written in the form $f = g - h$ with $g$ being strongly convex, and so the problem fits into our abstract I-ADCA framework. This allows us to derive the specialized inexact adaptive DC algorithm, denoted by I-ADCA-PDECO, for the PDE-constrained optimization problem \eqref{eq:mainprob}. Our analysis ensures well-posedness of the scheme, solvability and uniqueness of the exact subproblems, and finite-element error estimates that relate continuous and discrete controls. 

The rest of the paper is organized as follows. In Section~\ref{sec:prelim}, we recall some preliminary notions, which are broadly used in the subsequent material. Section~\ref{sec:i-ADCA} introduces the inexact adaptive DC algorithm I-ADCA and discusses some of its properties. In Section~\ref{conver-inexact}, we establish global convergence and convergence rates of the proposed algorithm under a version of the Polyak-\L ojasiewicz condition in Hilbert spaces. Section~\ref{sec:appl-PDECO} is devoted to optimal control of elliptic equations with $L^{1-2}$ sparsity-enhanced regularization and presents an algorithm for its solution based on applications of the inexact adaptive DCA developed above. In Section~\ref{finite-elem}, we justify a finite element scheme for the numerical solution of the optimal control problem  and present the results of numerical experiments. The concluding Section~\ref{conc} summarizes the main achievements of the paper and discusses some directions of our future research.\vspace*{-0.05in}

\section{Preliminaries}\label{sec:prelim}
\vspace*{-0.05in}

In this section, we recall for the reader's convenience several basic notions with the standard notation broadly used below.
In the framework of Hilbert spaces,  the symbols `$\rightarrow$' and `$\rightharpoonup$' signify the {\em strong and weak convergence}, respectively. Let $\la \cdot,\cdot\ra_{\mathcal{H}^*,\mathcal{H}}$ denote the {\em duality pairing} between a Hilbert space $\mathcal{H}$ and its continuous dual $\mathcal{H}^*$, i.e., 
$$
\la u^*,u\ra_{\mathcal{H}^*,\mathcal{H}}:=u^*(u) \text{ for every }
u^*\in \mathcal{H}^* \text{ and }u\in \mathcal{H}.
$$
The {\em diameter} of a set $\Omega\subset \mathcal{H}$ is defined as 
$$
\mathrm{diam}(\Omega):=\sup_{u,v\in \Omega}\|u-v\|_{\mathcal{H}}.
$$
Letting $u\in \mathcal{H}$ and $\Omega\subset \mathcal{H}$, denote 
$$
\mathrm{dist}(u;\Omega):=\inf_{v\in \Omega}\|u-v\|
$$ 
the {\em distance} from $u$ to $\Omega$. In $\mathcal{H}$, $B(u,r)$ and $\B(u,r)$ stand for the {\em open} and {\em closed ball} centered at $u$ with radius $r>0$, respectively, while $B$ and $\B$ signify in this way the open and closed unit ball in $\mathcal{H}$. Unless otherwise stated, all functions $f:\mathcal{H}\to \overline{\R}:=(-\infty,\infty]$ under consideration are {\em proper}, i.e., their {\em effective domain}
\begin{equation*}
\dom f := \{u\in \mathcal{H}\mid f(u) < \infty\}
\end{equation*}
is nonempty. An extended-real-valued function $f:\mathcal{H}\to \overline{\R}$ is called {\em lower semicontinuous} (l.s.c.) {\em at} $\overline{u}\in \dom f$ if for any sequence $u_k \to \overline{u}$, we have 
\begin{equation*}
\liminf_{k\to \infty} f(u_k) \ge f(\overline{u}).
\end{equation*}
The function $f$ is said to be lower semicontinuous if it is l.s.c. at any point in its domain. Further, $f$ is {\em convex} on ${\cal H}$ if 
\begin{equation*}
f\big(\lambda u + (1-\lambda)v\big)\le \lambda f(u) + (1-\lambda)f(v)\;\mbox{ for all }\;u,v\in \mathcal{H}\;\mbox{ and }\;\lambda\in [0,1].
\end{equation*}
The {\em modulus of strong convexity} of $f$ (in the sense of Polyak \cite{Polyak}), denoted as $\sigma_f$, is defined as the supremum of all $\sigma\ge 0$ for which the inequality
\begin{equation}\label{eq:strong-convex}
f\big(\lambda x +(1-\lambda)y\big) \le \lambda f(x) + (1-\lambda)f(y) -\dfrac{\sigma}{2}\lambda (1-\lambda)\|x-y\|^2
\end{equation}
holds whenever $x,y\in \mathcal{H}$ and $\lambda\in [0,1]$. It is said that $f$ is {\em strongly convex} if $\sigma_f >0$.

Given two Hilbert spaces $(\mathcal{H},\la \cdot,\cdot\ra_{\mathcal{H}})$ and $(\mathcal{K},\la \cdot,\cdot\ra_{\mathcal{K}})$, we endow the product space $\mathcal{H}\times \mathcal{K}$ with the  {\em inner product} defined by
\begin{equation*}
\la (u_1,v_1),(u_2,v_2)\ra_{\mathcal{H}\times \mathcal{K}}:= \la u_1,u_2\ra_{\mathcal{H}} + \la v_1,v_2\ra_\mathcal{K},
\end{equation*}
which induces the {\em product norm}
\begin{equation*}
\|(u,v)\|_{\mathcal{H}\times \mathcal{K}}:= \sqrt{\|u\|_{\mathcal{H}}^2 + \|v\|_{\mathcal{K}}^2}.
\end{equation*}
By $\mathscr{L}(\mathcal{H},\mathcal{K})$, we understand the normed space of linear bounded operators from $\mathcal{H}$ into $\mathcal{K}$.

\section{Inexact Adaptive DC Algorithm in Hilbert Spaces}\label{sec:i-ADCA}

Given a Hilbert space $\big(\mathcal{H},\la\cdot,\cdot\ra\big)$, a {\em DC minimization problem} on $\mathcal{H}$ is written as
\begin{equation}\label{eq:DC}
\inf_{u\in \mathcal{H}} \big[ f(u):= g(u)-h(u)\big],
\end{equation}
where $g$ and $h$ belong to the class $\Gamma (\mathcal{H})$ of proper l.s.c.  convex extended-real-valued functions on $\mathcal{H}$. In this case, $f$ is called a {\em DC function} while $g,h$ are called its {\em DC components}. 

A point $u\in \mathcal{H}$ is called a {\em critical point} of \eqref{eq:DC} if 
\begin{equation}\label{critical}
\partial g(u)\cap \partial h(u)\neq \varnothing,
\end{equation}
where the {\em subdifferential of convex analysis} is defined by
\begin{equation}\label{eq:convexsubdiff}
\partial g(u):=\big\{v\in \mathcal{H}\mid \la v, u'-u\ra \le g(u')-g(u)\, \text{ for all }\,u'\in \mathcal{H}\big\},\quad u\in \mathcal{H}.
\end{equation}

As is well known, DCA is a {\em descent method} without line search designed to find critical points of DC problems. Owing to its simplicity, efficiency, and robustness, DCA has proven to be highly effective for addressing large-scale nonsmooth nonconvex programs; see Section~\ref{sec:intro}. For problem \eqref{eq:DC}, it computes iteratively
\begin{equation}\label{DCA-ite}
\begin{cases}
v^k \in \partial h(u^k),\\
u^{k+1} \in \operatornamewithlimits{argmin}\limits_{u\in \mathcal{H}} \Big[g(u) - \la v^k,u\ra\Big],
\end{cases}
\end{equation}
where the {\em subproblem} of finding $u^{k+1}\in \mathrm{argmin}_{u\in \mathcal{H}} \Big[g(u) - \la v^k,u\ra\Big]$ is a convex program obtained from \eqref{eq:DC} by replacing the convex component $h$ with its affine minorization via taking the subgradient $v^k \in \partial h(u^k)$.\vspace*{0.05in}

The following {\em accelerated DCA} (ADCA), introduced in \cite{ijcai} in finite dimensions, incorporates {\em Nesterov’s acceleration scheme} into its construction to boost the performance of DCA. 

\begin{algorithm}[H]
\caption{ADCA: Accelerated DCA for solving \eqref{eq:DC}}
\textbf{Initialization:} {$z^0 = u^0=u^{-1} \in \dom f$, $q \in \mathbb{N}\cup \{0\}$, $t_0 = 1$, $k = 0$.}

\textbf{Repeat}
\begin{algorithmic}[1]\label{algo:adca}
\STATE Compute $t_{k+1} = \frac{1 + \sqrt{1 + 4 t_k^2}}{2}$.
\STATE Compute $z^k = u^k + \dfrac{t_k - 1}{t_{k+1}} \big(u^k - u^{k-1}\big)$.
\STATE If $f(z^k) \leq \displaystyle \max_{t = \max\{0,\,k-q\}, \ldots, k} f(u^t)$, then set $w^k = z^k$. Otherwise, set $w^k = u^k$.
\STATE Choose $v^k \in \partial h(w^k)$.
\STATE Compute 
\[
u^{k+1} \in \operatornamewithlimits{argmin}\limits_{u \in \mathcal{H}} \big[\, g(u) - \langle v^k, u \ra \,\big].
\]
\STATE Set $k \gets k + 1$.
\end{algorithmic}
\textbf{Until} {Stopping criteria}.
\end{algorithm}

To address {\em inexactness} in subgradient computations, we proceed---similarly to \cite{PHLH22} in finite dimensions---by replacing the exact subdifferential \eqref{eq:convexsubdiff} of a convex function $f:\mathcal{H}\to \overline{\R}$ at $u\in \dom f$   by its $\e$-{\em subdifferential} counterpart defined by 
\begin{equation}\label{eq:e-subdiff}
\partial_{\e}f(u):=\big\{v\in \mathcal{H}\mid \la v,u'-u\ra \le f(u')-f(u)+\e \, \text{ for all }\, u'\in \mathcal{H}\big\},
\end{equation}
with $\partial_0 f(u)$ reducing to \eqref{eq:convexsubdiff}; see \cite{nam,Zalinescu} for more details on \eqref{eq:e-subdiff}. We say that
$u_{\e}\in \mathcal{H}$ is an {\em $\e$-solution} to the problem $\inf_{u\in \mathcal{H}} f(u)$ and denote it as $u_{\e}\in \e\text{-}\mathrm{argmin}_{u\in \mathcal{H}}f(u)$ if 
\begin{equation*}
f(u_{\e}) \le \inf_{u\in \mathcal{H}} f(u) + \e.
\end{equation*}
In what follows, our standing assumptions on the data of \eqref{eq:DC} are: 
\begin{enumerate}
\item[{\bf(H1)}] $\dom h = \mathcal{H}$ and $h$ is bounded on bounded subsets of $\mathcal{H}$.
\item[{\bf(H2)}] $\inf_{u\in \mathcal{H}}f(u)>-\infty$.
\end{enumerate}

Now we are ready to design our underlying {\em inexact adaptive DCA} (abbr.\ I-ADCA) in which the main {\em adaptive step} (Step~4) is inspired by the inexact gradient step from \cite[Algorithm~1]{KMT-OMS}.

\begin{algorithm}[H]
\caption{I-ADCA: Inexact Adaptive DCA for Solving \eqref{eq:DC}}\label{algo:i-adca}
\textbf{Initialization:} $u^0 \in \dom f$, reduction factor $\gamma \in (0,1)$, and $\varepsilon_0 \in (0,1]$. Set $k=0.$

\noindent\textbf{Repeat}
\begin{algorithmic}[1]
\STATE Choose $w^k$ such that $f(w^k)\le f(u^k)$.
\STATE Compute $v^k \in \partial_{\varepsilon_k} h(w^k)$.
\STATE Compute an $\varepsilon_k$-solution $u^{k+1}$ of
\[
 \inf_{u \in \mathcal{H}} \big[\, g(u) - \langle v^k, u \ra \,\big].
\]
\STATE {\bf While} 
$$
\varepsilon_k > \frac{\sigma_g + \sigma_h}{32}\big\|u^{k+1} - w^k\big\|^2,
$$
set  $\varepsilon_k \gets \gamma\,\varepsilon_k$ and {\bf restart} from Step~2 with the new $\varepsilon_k$.
\STATE Choose $\varepsilon_{k+1} \in (0,\varepsilon_k]$ and set $k \gets k + 1$.
\end{algorithmic}
\noindent\textbf{Until} Stopping criteria.
\end{algorithm}

\begin{Remark}\rm 
Multiple selections of $w^k$ can satisfy Step~1 in Algorithm~\ref{algo:i-adca}. A particularly effective option to improve performance, illustrated numerically in Section~\ref{finite-elem}, is to combine {\em Nesterov-type acceleration} with the “look-back” step as in Algorithm~\ref{algo:adca} (when $q=0$), or to use other momentum schemes such as, e.g., {\em Polyak's heavy-ball method} \cite{polyak}.
\end{Remark}

Next we show that the steps of Algorithm~\ref{algo:i-adca} are {\em 
well-defined}. To proceed, recall the classical notion of {\em Fenchel conjugate} for $f\in \Gamma (\mathcal{H})$ given by
\begin{equation*}
f^* (v) := \sup\big\{\la u,v\ra - f(u)\mid u\in \mathcal{H}\big\}, \quad v\in \mathcal{H}.
\end{equation*}
It follows from \cite[Theorem~2.3.3]{Zalinescu} that $f^* \in \Gamma(\mathcal{H})$. Considering further a set-valued mapping/multifunction $S\colon{\cal H}\tto{\cal H}$, denote by
\begin{equation*}
\dom S:= \{u\in \mathcal{H}\mid S(u) \neq \varnothing\} \quad \text{and}\quad \mathrm{rge}\, S := \bigcup_{u\in \mathcal{H}} S(u)
\end{equation*}
its {\em domain} and {\em range}, respectively. First we check the well-definiteness of Steps~2 and 3 of Algorithm~\ref{algo:i-adca} in the following remark.

\begin{Remark}\rm $\,$

{\bf(i)} At each iteration, the nonemptiness of $\partial_{\e_k}h(w^k)$ is guaranteed by assumption (H1) and \cite[Theorem~2.4.4(iii)]{Zalinescu}. Thus, Step~2 of Algorithm~\ref{algo:i-adca} can always be carried out. 

{\bf(ii)} Regarding Step~3, observe that (H2) implies that 
\begin{equation*}
\inf_{v\in \mathcal{H}} \big[h^*(v) - g^* (v)\big] = \inf_{u\in \mathcal{H}} \big[g(u) - h(u)\big] = \inf_{u\in \mathcal{H}} f(u) >-\infty,
\end{equation*}
and therefore $\dom h^* \subset \dom g^*$. Then it follows from Step~2 in Algorithm~\ref{algo:i-adca} and \cite[Theorem~2.4.4(iii)-(iv)]{Zalinescu} that
\begin{equation*}
v^k \in \mathrm{rge}\, \partial_{\e_k}h = \dom \partial_{\e_k} h^*  = \dom h^* \subset \dom g^* = \dom \partial_{\e_k} g^*.
\end{equation*}
This ensures that Step 3 is well-defined.
\end{Remark}\vspace*{0.05in}

Now we show that the adaptive procedure in Step~4 {\em stops after finitely many iterations}.

\begin{Lemma}\label{lem:adap-finite} Let the function $h$ in \eqref{eq:DC} be continuous. At each iteration in Algorithm~{\rm\ref{algo:i-adca}}, if $w^k$ is not a critical point \eqref{critical} of \eqref{eq:DC}, then the adaptive process terminates after a finite number of steps. Consequently, we have the estimate
\begin{equation}\label{eq:adaptive-rule}
\varepsilon_k \le \frac{\sigma_g + \sigma_h}{32}\big\|u^{k+1} - w^k\big\|^2.
\end{equation}
\end{Lemma}
\begin{proof}
Fix  $k\in \N\cup \{0\}$ and suppose on the contrary that there exist sequences $\{v^{k,l}\}_{l\in \N}\subset \mathcal{H}$, $\{u^{k+1,l}\}_{l\in \N}\subset \mathcal{H}$, and $\{\varepsilon_{k,l}\}_{l\in \N}\subset \R$ decreasing to $0$ such that for all $k\in\N$ we have
\begin{equation}\label{eq:contra1}
\begin{cases}
v^{k,l} \in \partial_{\varepsilon_{k,l}}h(w^k),\\
u^{k+1,l} \in \varepsilon_{k,l}\text{-}\mathrm{argmin}\big[g(u)-\la v^{k,l},u\ra\big],\\
\varepsilon_{k,l}>\dfrac{\sigma_g + \sigma_h}{32}\big\|u^{k+1,l}-w^k\big\|^2.
\end{cases}
\end{equation}
The first relationship in \eqref{eq:contra1} ensures the boundedness of $\{v^{k,l}\}$ due to the boundedness of $\partial_{\varepsilon_{k,1}}h(w^k)$, which follows from \cite[Theorem~2.4.9]{Zalinescu} and the continuity of $h:\mathcal{H}\to \R$. Without loss of generality, assume that $v^{k,l}\rightharpoonup\ov$ as $l\to \infty$, which together with the first inclusion in \eqref{eq:contra1} yields $\ov\in \partial h(w^k)$.
The last condition in \eqref{eq:contra1} clearly gives us $u^{k+1,l}\to w^k$ as $l\to \infty$. Moreover, from \eqref{eq:contra1} we have the inequalities
\begin{equation}\label{eq:before-lime}
\la v^{k,l},u-u^{k+1,l}\ra \le g(u) - g\big(u^{k+1,l}\big) + \varepsilon_{k,l}\quad \text{ for all } u\in \mathcal{H} \text{ and }\;l\in \N.
\end{equation}
Taking further into account the lower semicontinuity of $g$ and the weak convergence $v^{k,l}\rightharpoonup\ov$ as well as the strong one $\big(u^{k+1,l},\varepsilon_{k,l}\big)\to (w^k,0)$, we pass to the limits as $l\to \infty$ in \eqref{eq:before-lime} and arrive at the inequality
\begin{equation*}
\la\ov,u-w^k\ra \le g(u) - g(w^k)\quad \text{ for all } u\in \mathcal{H},
\end{equation*}
i.e., $\ov\in \partial g(w^k)$. This shows that $w^k$ is a critical point of problem \eqref{eq:DC}, a contradiction that completes the proof of the lemma.
\end{proof}

In the finite-dimensional setting, it is proved in \cite[Theorem~3(i)]{DCA97} that the DCA update rule \eqref{DCA-ite} generates a monotonically decreasing sequence $\{f(u^k)\}$ of objective values for problem \eqref{eq:DC}. Leveraging the {\em adaptive mechanism} (Step~4 in Algorithm~\ref{algo:i-adca}), we now establish a similar property in the presence of {\em inexactness} in the Hilbert space framework. The property of strong convexity \eqref{eq:strong-convex} is significantly utilized below.
\vspace*{0.05in}

\begin{Lemma}\label{lem:converge-stand-2}
Let $g,h\in \Gamma (\mathcal{H})$, and let the function $h$ be continuous. Suppose that Algorithm~{\rm\ref{algo:i-adca}} generates sequences $\{\e_k\}$, $\{u^k\}$, $\{v^k\}$, and $\{w^k\}$, where $w^k$ is not a critical point of \eqref{eq:DC} for each $k\in \N \cup \{0\}$. Then the following assertions hold.

{\bf(i)} For any $k\in \N \cup \{0\}$, we have the estimates \begin{equation}\label{eq:converge-stand-3}
\dfrac{\sigma_g + \sigma_h}{8} \big\|u^{k+1}-w^{k}\big\|^2\le f(w^k) - f(u^{k+1})\le f(u^k) - f(u^{k+1}).
\end{equation}

{\bf(ii)} If either $g$ or $h$ is strongly convex, then 
\begin{equation}\label{eq:sumfin}
\displaystyle\sum\limits_{k=0}^{\infty} \big\|u^{k+1} - w^{k} \big\|^2 <\infty,
\end{equation}
which implies that $\lim\limits_{k\rightarrow \infty}\big\|u^{k+1} - w^{k} \big\| =0$ and $\lim\limits_{k\to \infty}\e_k =0$.
\end{Lemma}
\begin{proof}
We begin with verifying \eqref{eq:converge-stand-3} for a fixed number $k\in \N \cup \{0\}$. Algorithm~\ref{algo:i-adca} asserts that $v^k \in \partial_{\e_k} h(w^k)$, which together with \eqref{eq:e-subdiff} and \eqref{eq:strong-convex} yields
\begin{equation*}
\begin{aligned}
\Big\la v^k, \dfrac{1}{2}(u^{k+1}-w^k)\Big\ra&=\Big\la v^k, \Big(\dfrac{1}{2}u^{k+1}+\dfrac{1}{2}w^k\Big) - w^k\Big\ra \\
&\le h\Big(\dfrac{1}{2}u^{k+1}+\dfrac{1}{2}w^k\Big) - h(w^k)+\ve_k\\
&\le -\dfrac{\sigma_h}{8}\|u^{k+1}-w^k\|^2+ \dfrac{1}{2}h(u^{k+1}) +\dfrac{1}{2} h(w^k)-h(w^k) + \e_k \\
&= -\dfrac{\sigma_h}{8}\|u^{k+1}-w^k\|^2+ \dfrac{1}{2}h(u^{k+1}) -\dfrac{1}{2} h(w^k) + \e_k.
\end{aligned}
\end{equation*}
It follows therefore that 
\begin{equation}\label{eq:41}
\la v^k,u^{k+1}-w^k\ra + \dfrac{\sigma_h}{4}\big\|u^{k+1}-w^k\big\|^2 \le h(u^{k+1})-h(w^k)+2\e_k.
\end{equation}
Furthermore, Algorithm~\ref{algo:i-adca} tells us that $g(u^{k+1})-\la v^k,u^{k+1}\ra \le g(u)-\la v^k,u\ra + \e_k$ for all $u\in \mathcal{H}$, i.e., $v^k\in \partial_{\e_k}g(u^{k+1})$. Similarly to \eqref{eq:41}, we get
\begin{equation*}
\la v^k,w^{k}-u^{k+1}\ra + \dfrac{\sigma_g}{4}\big\|w^k-u^{k+1}\big\|^2 \le g(w^k)-g(u^{k+1})+2\e_k,
\end{equation*}
which  being added to \eqref{eq:41} implies that
\begin{equation}\label{eq:42}
\dfrac{\sigma_g + \sigma_h}{4}\big\|u^{k+1}-w^k\big\|^2\le f(w^k) - f(u^{k+1})+4\e_k.
\end{equation}
Employing \eqref{eq:adaptive-rule}, we deduce from \eqref{eq:42} that
\begin{equation*}
\begin{aligned}
\dfrac{\sigma_g + \sigma_h}{4}\big\|u^{k+1}-w^k\big\|^2 &\le f(w^k) - f(u^{k+1})+4\e_k\\
&\le f(w^k) - f(u^{k+1}) + \dfrac{\sigma_g + \sigma_h}{8}\|u^{k+1}-w^k\|^2.
\end{aligned}
\end{equation*}
Consequently, Algorithm~\ref{algo:i-adca} ensures that 
\begin{equation*}
\dfrac{\sigma_g + \sigma_h}{8}\big\|u^{k+1}-w^k\big\|^2 \le f(w^k) - f(u^{k+1}) \le f(u^k) - f(u^{k+1}),
\end{equation*}
which thus justifies assertion (i) .
    
To verify (ii), assume that either $g$ or $h$ is strongly convex, i.e., $\sigma_g +\sigma_h>0$. Summing up \eqref{eq:converge-stand-3} over $k=0,\ldots,N$ with $N\in \N$ yields 
\begin{equation*}
\begin{aligned}
\dfrac{\sigma_g + \sigma_h}{8} \displaystyle\sum\limits_{k=0}^N \big\| u^{k+1} - w^{k} \big\|^2 &\le \displaystyle\sum\limits_{k=0}^N \left[f(u^k) - f(u^{k+1}) \right] \\
&=f(u^0) - f(u^{N+1})\\
&\le f(u^0) - \inf_{u\in \mathcal{H}}f(u) .
\end{aligned}
\end{equation*}
Since $\sigma_g +\sigma_h>0$, we get in addition that
\begin{equation*}
\displaystyle\sum\limits_{k=0}^N \big\| u^{k+1} - w^{k} \big\|^2 \le \dfrac{f(u^0) - \inf_{u\in \mathcal{H}}f(u)}{(\sigma_g + \sigma_h)/8}.
\end{equation*}
Passing to the limit as $N\rightarrow \infty$ leads us to
\begin{equation*}
\displaystyle\sum\limits_{k=0}^{\infty} \big\| u^{k+1} - w^{k} \big\|^2 < \infty,
\end{equation*}
which readily implies that $\|u^{k+1}-w^k\|\to 0$ as $k\to \infty$. Then it follows from \eqref{eq:adaptive-rule} that $\e_k\to~0$ thus completing the proof of the lemma.
\end{proof}

Our first theorem addresses convergence of a {\em subsequence} of $\ve_k$-solutions of the {\em subproblems} in Algorithm~\ref{algo:i-adca} to a {\em critical point} of \eqref{eq:DC}.

\begin{Theorem}\label{theo:converge-stand-1} Assume in the setting of 
Lemma~{\rm\ref{lem:converge-stand-2}(ii)} that the sequence $\{u^k\}$ is bounded. Then the sequence $\{v^k\}$ is bounded as well, and for any subsequence $\big\{u^{k_j} \big\}$ of $\big\{ u^{k} \big\}$ strongly converging to some $\ou$, the limiting point $\ou$ is a critical one for \eqref{eq:DC}.
\end{Theorem}
\begin{proof}
If $\{u^k\}$ is bounded, then Lemma~\ref{lem:converge-stand-2}(ii) yields the boundedness of $\{w^k\}$. Then it follows from Step~2 in Algorithm~\ref{algo:i-adca}, assumption (H1), and \cite[Theorem~2.4.13]{Zalinescu} that the sequence $\{v^k\}$ is bounded as well. 

Let $\{u^{k_j}\}$ be a subsequence of $\left\{u^{k} \right\}$ that converges to some $\ou\in \mathcal{H}$. Then it follows from Lemma~\ref{lem:converge-stand-2}(ii) that $\lim\limits_{j\rightarrow \infty} w^{k_j-1} =\ou$. By the boundedness of $\{v^{k_j-1} \}$ in $\mathcal{H}$, we may assume its weak convergence to some $\ov\in \mathcal{H}$. Algorithm~\ref{algo:i-adca} gives us the subgradient estimates
\begin{equation*}
\la v^{k_j-1},u-w^{k_j-1}\ra \le h(u)-h(w^{k_j-1})+\e_{k_j-1}\;\mbox{ for all }\;j\in \N\;\mbox{ and }\;u\in \mathcal{H}.
\end{equation*}
Passing there to the limit as $j\rightarrow\infty$ and using the convergence $w^{k_j-1} \rightarrow\ou$ and $v^{k_j-1} \rightharpoonup\ov$ lead us to the relationships
\begin{equation*}
\begin{aligned}
\la\ov,u-\ou\ra &\le h(u) + \limsup_{j\to \infty} \big[-h(w^{k_j-1})\big]\\
&= h(u) -\liminf_{j\to \infty} h(w^{k_j-1}) \le h(u) - h(\ou),
\end{aligned}
\end{equation*}
and therefore $\ov\in \partial h(\ou)$. On the other hand, recall from Step~3 in Algorithm~\ref{algo:i-adca} that 
\begin{equation}\label{conse-s5}
g(u^{k_j}) - \la v^{k_j-1},u^{k_j}\ra \le g(u)-\la v^{k_j-1},u\ra+\e_{k_j-1}\;\text{ for all }\;j\in \N\;\text{ and } u\in \mathcal{H}.
\end{equation}
By letting $j\rightarrow \infty$ in \eqref{conse-s5}, it follows that
\begin{equation}\label{eq:limsupG}
\limsup_{j\rightarrow \infty} g(u^{k_j}) \le  \la\ov,\ou-u\ra +g(u) \quad \text{for any}\quad u\in \mathcal{H},
\end{equation}
and hence $\limsup_{j\rightarrow \infty} g(u^{k_j}) \le g(\ou)$. The latter estimate together with the imposed lower semicontinuity of $g$ implies that $g(u^{k_j}) \rightarrow g(\ou)$ as $j\rightarrow \infty$. Consequently, \eqref{eq:limsupG} yields $$g(\ou)\le \la\ov,\ou-u\ra + g(u) \quad \text{for all}\quad u\in \mathcal{H},$$ i.e., $\ov\in \partial g(\ou)$. This shows that $\ou$ is a critical point of \eqref{eq:DC} and thus completes the proof.
\end{proof}\vspace*{-0.1in}

\section{Global Convergence and Convergence Rates of Inexact Adaptive DCA}\label{conver-inexact}

In this section, we establish global convergence with explicit convergence rates for Algorithm~\ref{algo:i-adca} under some additional assumptions. First, we recall the needed notions from infinite-dimensional variational analysis; see, e.g., \cite{Mordukhovich06,Mordukhovich24} and the references therein. For a proper function $f:\mathcal{H} \to \overline{\R}$, the (Fr\'echet) {\em regular subdifferential} of $f$ at $\ou\in\dom f$ is  
\begin{equation}\label{FrechetSubdifferential}
\partial_F f(\ou):=\left\{v\in \mathcal{H}\;\Big|\;\liminf\limits_{u\to \overline u}\frac{\varphi(u)-\varphi(\ou)-\langle v,u-\ou\rangle}{\|u-\ou\|}\ge 0 \right\},
\end{equation}
while the (Mordukhovich) {\em limiting subdifferential} of $f$ at $\ou$ is defined by
\begin{equation}\label{MordukhovichSubdifferential}
\partial_M f(\overline{u}):=\big\{\ov \in \mathcal{H}\; \big|\;\exists \text{ sequences }v_k \rightharpoonup \ov,\; u_k \xrightarrow{f} \ou \text{ with }v_k \in \partial_F f(u_k),\;k\in \N\big\},
\end{equation}
where $u_k\xrightarrow{f}\ou$ signifies that $u_k\rightarrow\overline{u}$ and $f(u_k)\rightarrow f(\overline{u})$ as $k\to \infty$. Both constructions $\partial_F f(\ou)$ and $\partial_M f(\overline{u})$ reduce to the classical gradient $\{\nabla f(\ou)\}$ when $f$ is of class ${\cal C}^1$ around $\ou$. They also agree with the convex subdifferential \eqref{eq:convexsubdiff} for convex functions $f$. A vector $\ou\in \mathcal{H}$ is called a {\em limiting/(M)ordukhovich-stationary point} of $f$ if 
\begin{equation}\label{defi:M-crit}
0\in \partial_M f(\ou).
\end{equation}

A central role in our global convergence and convergence rate results is played by the following version of the (generalized) {\em Polyak-\L ojasiewicz property}.

\begin{Definition}\label{defi:Loja}   
A proper function $f:\mathcal{H} \rightarrow \overline{\R}$ has the {\sc Polyak-\L ojasiewicz property} if:

{\bf(i)} $f|_{\mathrm{dom} f}$ is continuous on $\dom f$.

{\bf(ii)} For each $M$-stationary point $\ou$ of $f$, there exist $C>0$, $\varepsilon>0$, and $\theta \in [0,1)$ such that 
\begin{equation}\label{pl}
\left| f(u) - f(\ou) \right|^{\theta} \le C \|v\|\;\mbox{ for all }\;u\in B(\ou, \varepsilon)\;\mbox{ and }\;v\in \partial_M f(u),
\end{equation}

with the convention that $0^0 := 0$.
\end{Definition}

When $\theta=1/2$, property \eqref{pl} was introduced by Polyak \cite{Polyak} for functions of class ${\cal C}^{1,1}$ (i.e., continuously differentiable with Lipschitzian gradients) in Hilbert spaces. Polyak's motivation and main result \cite[Theorem~4]{Polyak} were to establish linear convergence of the gradient descent method for ${\cal C}^{1,1}$ functions in the Hilbert space setting. Independently, \L ojasiewicz \cite{Lojasiewicz} defined a version of \eqref{pl} for analytic functions in finite-dimensions for arbitrary $\theta\in(0,1)$ by using tools of semialgebraic geometry without any applications to optimization. Much later, Kurdyka \cite{Kurdyka} extended the semialgebraic approach by \L ojasiewicz to the general class of o-minimal structures. Extensions of \eqref{pl} to nonsmooth functions via various subgradient mappings were proposed under different names and applied to optimization problems in, e.g.,  \cite{AFV18,amp24,AB09,bmmn25,KMT-OMS,Mordukhovich24} and many other publications in finite-dimensional spaces. Hilbert space extensions were given in \cite{BDLM09,LojaHilbert11} and the references therein. In the aforementioned publications, the reader can find broad classes of functions for which the Polyak-\L ojasiewicz property and its generalizations hold.\vspace*{0.05in}

Recall further that $\ou\in \mathcal{H}$ is a {\em strong cluster point} of a sequence $\{u_k\}\subset \mathcal{H}$ if there exists a subsequence of $\{u_k\}$ that strongly converges to $\ou$. Before deriving the main result of this section, we present two lemmas of their own interest.

\begin{Lemma}\label{lem:iadca-rate-lem2}
In the setting of Theorem~{\rm\ref{theo:converge-stand-1}}, let $\{u^k\}$, $\{v^k\}$, and $\{w^k\}$ be generated by Algorithm~{\rm\ref{algo:i-adca}}. Assume in addition that $f=g-h$ is l.s.c., $\{u^k\}$ is bounded, and the set $\Omega$ of strong cluster points of $\{u^k\}$ is nonempty. We have the following assertions.

{\bf(i)} $f(u^k) \searrow f(\ou)$ as $k\to \infty$ for any $\ou\in \Omega$. Thus $f$ is constant on $\Omega$ with the value $\bar f$. 

{\bf(ii)} Suppose that $f$ has the Polyak-\L ojasiewicz property, $h$ is differentiable, and $\Omega$ is compact. Then there exist $C,\varepsilon>0$ and $\theta \in [0,1)$ such that when  $\mathrm{dist}(u;\Omega) \le \varepsilon$, we have the estimate
\begin{equation*} 
\left| f(u) - \bar f\right|^{\theta} \le C \|v\|\; \text{ whenever }\;v\in \partial_M f(u).
\end{equation*}

{\bf(iii)} If $f$ has the Polyak-\L ojasiewicz property while $g$ is differentiable and $\Omega$ is compact, then there exist $C,\varepsilon>0$ and $\theta \in [0,1)$ such that when  $\mathrm{dist}(u;\Omega) \le \varepsilon$, we have
\begin{equation*}
\left| f(u) - \bar f \right|^{\theta} \le C \|v\|\; \text{ whenever }\;v\in \partial_M (-f)(u).
\end{equation*}
\end{Lemma}
\begin{proof}
To verify (i), we get from Lemma~\ref{lem:converge-stand-2}(i) that 
$\{f(u^k)\}$ is a decreasing sequence. Then it follows from assumption (H2) that $\{f(u^k)\}$ converges. Pick further some $\ou\in \Omega$ and show that $f(u^k)\to f(\ou)$ as $k\to \infty$. By the definition of $\Omega$, there exists a subsequence $\{u^{k_j}\}$ that strongly converges to $\ou$. Since $\{u^{k}\}$ is bounded, Theorem~\ref{theo:converge-stand-1} verifies the boundedness of $\{v^{k_j-1}\}$. Arguing as in the proof of Theorem~\ref{theo:converge-stand-1} leads us to the strong convergence 
\begin{equation*}
g(u^{k_j}) \to g(\overline{u}) \text{ as }j\to \infty.
\end{equation*}
Then the lower semicontinuity of $h$ ensures that
\begin{equation*}
\begin{split}
\limsup\limits_{j\rightarrow \infty} f(u^{k_j}) & =  \limsup\limits_{j\rightarrow \infty} \big[ g(u^{k_j}) - h(u^{k_j}) \big] \\
& = g(\ou)-  \liminf\limits_{j\rightarrow \infty} h(u^{k_j}) \\
& \le g(\ou) - h(\ou) = f(\ou),
\end{split}
\end{equation*}
and thus we have $f(u^{k_j})\to f(\ou)$ by the lower semicontinuity of $f$. The convergence of $\{f(u^k)\}$ yields $f(u^k)\to f(\ou)$, and therefore $f$ has the same value $\bar f$ on $\Omega$ as claimed in (i).

To verify (ii) under the additional assumptions imposed, first we show that $\Omega$ is contained in the set of $M$-stationary points \eqref{defi:M-crit} of $f$. Indeed, Theorem~\ref{theo:converge-stand-1} implies that any $\ou\in \Omega$ is a critical point of $f=g-h$, i.e., $\nabla h(\ou) \in\partial g(\ou)$. On the other hand, the subdifferential sum rule taken from \cite[Proposition~1.107(i)]{Mordukhovich06}
ensures the equalities
\begin{equation*}
\partial_F f(\ou) = \partial_F(g-h)(\ou) = \partial_F g(\ou) - \nabla h(\ou) = \partial g(\ou) - \nabla h(\ou),
\end{equation*}
which thus imply that $0 \in \partial_F f(\ou)\subset \partial_M f(\ou)$ confirming that $\ou$ is an $M$-stationary point of $f$. 

Employing further the Polyak-\L ojasiewicz property of $f$, we find for every $u' \in \Omega$ (which, as above, is also an $M$-stationary point of $f$) numbers $C',\varepsilon'>0$ and $\theta'\in [0,1)$ such that 
\begin{equation}\label{eq:L}
\left| f(u) - \bar{f} \right|^{\theta'} \le C' \|v\|\quad \text{ whenever } u\in B(u', \varepsilon')\text{ and } v \in \partial_M f(u).
\end{equation}
By the continuity of $f|_{\mathrm{dom} f}$ with respect to $\dom f$, we may shrink $\varepsilon'>0$ if needed so that 
$$
|f(u)-\bar f| = |f(u) - f(u')|<1\quad \text{for all }u\in \dom f \cap B(u',\varepsilon').
$$ 
The compactness of $\Omega$ allows us to pick finitely many points $u'_i\in \Omega$ $(i=1,\ldots,l)$, corresponding to $C'_i,\varepsilon'_i>0$ and $\theta'_i\in [0,1)$, such that $\Omega \subset \bigcup_{i=1}^l B(u'_i,\varepsilon'_i)$ and 
\begin{equation}
\begin{cases}
\left| f(u) - \bar f \right|^{\theta'_i} \le C'_i \|v\|\quad \text{ for all }\;u\in B(u'_i,\varepsilon'_i),\;v\in\partial_M f(u),\;\text{ and }\;i\in\{1,\ldots,l\},\\
|f(u)-\bar f| < 1\quad \text{for all }\, u\in \dom f\cap B(u'_i,\varepsilon'_i)\;\mbox{ and }\; i\in\{1,\ldots,l\}.
\end{cases}\label{eq:use-comp}
\end{equation}
Furthermore, it follows from the compactness of $\Omega$ that there exists $\varepsilon \in (0,1)$ ensuring that $\Omega + \varepsilon \B\subset \bigcup_{i=1}^l B(u'_i,\varepsilon'_i)$, which readily yields
\begin{equation}\label{eq:distin}
\big\{u\in \mathcal{H}\mid \mathrm{dist}(u;\Omega)\le\varepsilon\big\} \subset \Omega + \varepsilon \B\subset \bigcup_{i=1}^l B(u'_i,\varepsilon'_i).
\end{equation}
By letting $C:=\max_{i=1,\ldots,l}C'_i$ and $\theta:= \max_{i=1,\ldots,l}\theta'_i$, selecting $u$ with $\mathrm{dist}(u;\Omega)\le \varepsilon$ and $v\in \partial_M f(u)$, and then employing \eqref{eq:use-comp} and \eqref{eq:distin} lead us to 
\begin{equation*}
|f(u) - \bar f|^{\theta} \le |f(u)-\bar f|^{\theta'_i} \le C'_i \|v\| \le C\|v\|,
\end{equation*}
where $i\in \{1,\ldots,l\}$ is such that $u\in B(u'_i,\varepsilon'_i)$, and thus
assertion (ii) is justified. The verification of (iii) is similar, which completes the proof of the lemma. 
\end{proof}

The next lemma is a direct consequence of the fundamental Ekeland's variational principle (see, e.g., \cite[Theorem~2.26]{Mordukhovich06}) and the classical Moreau-Rockafellar's subdifferential sum rule of convex analysis in Hilbert spaces presented, e.g., in \cite[Theorem~3.48]{nam}.

\begin{Lemma}\label{lem:esub-appro}
Take $f\in \Gamma (\mathcal{H})$ and $\e>0$. For any point $u\in \dom f$ and $v\in \partial_{\e}f(u)$, there exists $w\in \B\big(u,\sqrt{\e}\big)$ such that $v\in \partial f (w) + \sqrt{\e}\B$. 
\end{Lemma}

Now we formulate the major assumptions, which are used---alternatively---in the main convergence results for Algorithm~\ref{algo:i-adca} given below.\vspace*{0.05in}

{\bf(A1)}: $f|_{\mathrm{dom} f}$ is locally Lipschitz with respect to $\dom f$, and $f$ possesses the Polyak-\L ojasiewicz property.

{\bf(A2)}: $g$ is of class ${\cal C}^{1,1}$.

{\bf(A3)}: $h$ is of class ${\cal C}^{1,1}$.

{\bf(A4)}: Any subsequence of $\{u^k\}$ generated by Algorithm~\ref{algo:i-adca} has a strongly convergent subsequence, and the set $\Omega$ of strong cluster points of $\{u^k\}$ is compact.  

\begin{Theorem}\label{theo:adca-theo2} Let {\rm(A1)} and {\rm(A4)} be satisfied. In the setting of Lemma~{\rm\ref{lem:iadca-rate-lem2}} with $C,\ve>0$ and $\th\in [0,1)$, assume that either {\rm(A2)} holds for $g$, or {\rm(A3)} holds for $h$. Consider a variant of Algorithm~{\rm\ref{algo:i-adca}} in which we pick in Step~{\rm 2} a vector $v^k\in \partial_{\varepsilon_k^{1/\theta}}h(w^k)$ if {\rm(A2)} holds, and pick in Step~{\rm 3} an $\varepsilon_{k}^{1/\theta}$-solution $u^{k+1}$ if {\rm {\rm{}(A3)}} is satisfied. Then we have the following assertions:

{\bf (i)} When $\theta=0$, the sequence $\{f(u^k)\}$ converges to $\bar{f}$ in a finite number of steps.

{\bf(ii)} When $\theta \in (0,1/2]$, the sequence $\{f(u^k)\}$ converges linearly to $\bar f$.

{\bf(iii)} When $\theta \in (1/2,1)$, there exists $\eta>0$ such that 
\begin{equation*}
f(u^k) - \bar f\le \eta k^{\frac{1}{1-2\theta}} \ \text{ whenever } k\in \N \text{ is sufficiently large.}
\end{equation*}
\end{Theorem}
\begin{proof}
For each $k\in \N\cup\{0\}$, we have $\varepsilon_{k}^{1/\theta}\le \varepsilon_{k}$ implying that
$$
\begin{aligned}
\partial_{\varepsilon_k^{1/\theta}}h(w^k)&\subset \partial_{\varepsilon_k}h(w^k),\\
\varepsilon_k^{1/\theta}\text{-argmin}_{u\in \mathcal{H}}[g(u) - \la v^k,u\ra] &\subset \varepsilon_k\text{-argmin}_{u\in \mathcal{H}}[g(u) - \la v^k,u\ra].
\end{aligned}
$$
This tells us that the algorithm formulated in the theorem is a particular case of Algorithm~\ref{algo:i-adca}, and thus all the conclusions established above
apply to the generated sequences $\{u^k\},\{w^k\}$, and $\{v^k\}$. Since any subsequence of $\{u^k\}$ admits a strongly convergent subsequence by (A4), and since $\mathrm{dist}(\cdot;\Omega):\mathcal{H}\to \R$ is a continuous function, it is easy to see that 
$$
\mathrm{dist}(u^k;\Omega) \to 0 \,\text{ as }\, k\to \infty.
$$ 
Consequently, the convergence $\|u^{k+1}-w^k\|\to 0$ obtained in
Lemma \ref{lem:converge-stand-2}(ii) and the Lipschitz continuity of $\mathrm{dist}(\cdot;\Omega)$ yield the convergence 
\begin{equation}\label{eq:wk-to-Ome}
\mathrm{dist}(w^k;\Omega) \rightarrow 0\, \text{ as }\, k\to \infty.
\end{equation}
It follows from Theorem~\ref{theo:converge-stand-1} and the boundedness of $\{u^k\}$ that
\begin{equation}\label{eq:omega-domf}
\Omega \subset \dom \partial g \cap \dom \partial h\subset \dom g \cap \dom h = \dom f.
\end{equation}
The remainder of the proof is split into the following two cases.\vspace*{0.05in}

{\bf Case~1:} {\em Assumption {\rm(A2)} holds}. Fix $x\in \Omega$ and find by using \eqref{eq:omega-domf}, (A1), and (A2) positive numbers $\ell_x,r_x >0$ ensuring the estimates 
\begin{equation}\label{eq:adca-rate-1}
\begin{split}
\big|f(u)-f(v)\big| \le \ell_x \|u-v\|\quad &\text{ for all } u,v\in B(x,r_x)\cap \dom f,\\
\left\| \nabla g(u) - \nabla g(v) \right\| \le \ell_x \| u - v \|\quad &\text{ for all } u,v \in B(x,r_x).
\end{split}
\end{equation}
The compactness of $\Omega$ allows us to extract finitely many points $x_1,\ldots,x_m \in \Omega$ corresponding to the sequences of positive parameters $\{\ell_i\}_{i=1}^m$ and $\{r_i\}_{i=1}^m$ with $\Omega \subset \bigcup_{i=1}^m B\big(x_i, r_{i} /8 \big)$. Define 
\begin{equation}\label{eq:lr}
\ell:= \max\big\{1,\ell_{i} \mid i=1,\ldots,m\big\}\; \text{ and }\; r := \min\big\{ \e/2, r_{i} /8 \mid i=1,\ldots,m \big\},
\end{equation}
where $\e>0$ is taken from (A1) and Definition~\ref{defi:Loja}. It follows from \eqref{eq:wk-to-Ome} and Lemma~\ref{lem:converge-stand-2}(ii) that 
\begin{equation}\label{eq:choosek}
\max\big\{\mathrm{dist}(w^{k};\Omega),\;\|u^{k+1}-w^k\|,\;\e_k^{1/2}\big\} < r\quad \text{ for all } k\in \N \text{ sufficiently large},
\end{equation}
which yields in turn the inclusions
\begin{equation}\label{eq:wk-in-ball}
w^k \in r\B + \Omega \subset r\B + \bigcup_{i=1}^m B(x_i,r_i/8) \subset \bigcup_{i=1}^m B(x_i,r_i/4),
\end{equation}
where the last one is due to the choice of $r$ in \eqref{eq:lr}. As a consequence of \eqref{eq:choosek} and \eqref{eq:wk-in-ball}, for each large $k\in \N$ there are $x_{i(k)}\in \Omega$, $i(k)\in \{1,\ldots,m\}$, with 
$$
\{u^{k+1},w^k\}\subset B\big(x_{i(k)},3r
_{i(k)}/8\big).
$$ 
Employing Lemma~\ref{lem:esub-appro}, we deduce from $v^k \in \partial_{\e_k}g(u^{k+1})\cap \partial_{\e_k^{1/\theta}}h(w^k)$ the existence of vectors $x^{k+1}\in \B\big(u^{k+1},\e_k^{1/2}\big)$ and $y^k \in \B(w^k,\e^{1/2\theta}_k\big) \subset \B(w^k,\e_k^{1/2}\big)$ for which
\begin{equation}\label{eq:use-esub-appro}
\begin{array}{ll}
v^k &\in\big[\nabla g(x^{k+1})+\e_k^{1/2}\B\big]\cap\big[\partial h(y^k)+\e^{1/2\theta}_k\B\big]\\
&\subset\big[\nabla g(x^{k+1})+\e_k^{1/2}\B\big]\cap\big[ \partial h(y^k)+\e^{1/2}_k\B\big].
\end{array}
\end{equation}
By $\big\{u^{k+1},w^k\big\} \subset B\big(x_{i(k)},3r_{i(k)}/8\big)$, it follows from \eqref{eq:lr} and \eqref{eq:choosek} that
\begin{equation}\label{eq:ytox}
\begin{cases}
\big\|x^{k+1}-x_{i(k)}\big\|\le \|x^{k+1}-u^{k+1}\| + \big\|u^{k+1}-x_{i(k)}\big\| \le \e_k^{1/2} + 3r_{i(k)}/8 < r_{i(k)}/2, \\
\big\|y^k-x_{i(k)}\big\|\le \|y^k-w^k\| + \big\|w^k-x_{i(k)}\big\| \le \e_k^{1/2} + 3r_{i(k)}/8 < r_{i(k)}/2.
\end{cases}
\end{equation}
We also have the estimates
\begin{equation}\label{eq:xkyk}
\begin{split}
\big\|x^{k+1}-y^k\big\| &\le \big\|x^{k+1}-u^{k+1}\big\| + \big\|u^{k+1}-w^k\big\| + \big\|w^k - y^k\big\|\le 2\e_k^{1/2} + \big\|u^{k+1}-w^k\big\|.
\end{split}
\end{equation}
The inclusions in \eqref{eq:use-esub-appro} imply by the Lipschitzian conditions in \eqref{eq:adca-rate-1} and the relationships in \eqref{eq:ytox} and \eqref{eq:xkyk} that 
\begin{equation*}
v^k \in \left[\nabla g(y^k) + \left(\big(2\ell+1\big)\e_k^{1/2} + \ell\big\|u^{k+1}-w^k\big\|\right)\B \right] \cap \left[\partial h(y^k) + \e^{1/2}_k\B\right] \ \ \text{ for all large }\;k\in \N.
\end{equation*}
This gives us a subgradient $b^k\in \partial h(y^k)$ satisfying 
\begin{equation}\label{eq:getbk}
\begin{split}
\big\|b^k-v^k\big\| &\le \e_k^{1/2},\\
\big\|b^k-\nabla g(y^k)\big\| &\le \|b^k - v^k\| + \|v^k - \nabla g(y^k)\|\le 2\big(\ell+1\big)\e_k^{1/2} + \ell\big\|u^{k+1}-w^k\big\|,\\
b^k-\nabla g(y^k) &\in \partial_F  h(y^k)-\nabla g(y^k) =\partial_F(h-g)(y^k)=\partial_F(-f)(y^k) \subset \partial_M(-f)(y^k).
\end{split}
\end{equation}
Deduce further from \eqref{eq:lr}, \eqref{eq:choosek}, and the nonexpansive property of $\mathrm{dist}(\cdot;\Omega)$ that 
$$
\mathrm{dist}(y^k;\Omega)\le \mathrm{dist}(w^k;\Omega)+\e_k^{1/2} \le \e.
$$
Taking \eqref{eq:getbk} into account and applying the result of 
Lemma~\ref{lem:iadca-rate-lem2}(iii) leads us to the inequalities
\begin{equation}\label{eq:Fyk}
\big|f(y^{k})-\bar{f}\big|^{\theta}\le C \big\|b^k-\nabla g(y^k)\big\| \le C\left[2\big(\ell+1\big)\e_k^{1/2} + \ell\big\|u^{k+1}-w^k\big\|\right].
\end{equation}

In order to achieve the claimed convergence rates of the algorithm, we need to pass from the information of $f$ at $y^k$ to the iterate $u^{k+1}$ in \eqref{eq:Fyk}. To this end, combining \eqref{eq:adaptive-rule}, \eqref{eq:converge-stand-3}, \eqref{eq:adca-rate-1}, \eqref{eq:Fyk}, and the decrease of $\{f(u^k)\}$ to $\bar f$ from Lemma~\ref{lem:iadca-rate-lem2} tells us that
\begin{equation}\label{eq:beforerho}
\begin{split}
\left[f(u^{k+1})-\bar f\right]^{\theta} &= \left[f(u^{k+1})-f(w^k)+ f(w^k)-f(y^k) + f(y^{k})-\bar f\right]^{\theta} \\
&\le \left[0+ \big|f(w^k)-f(y^k)\big| + \big|f(y^{k})-\bar f\big|\right]^{\theta}\\
&\le \big|f(w^k)-f(y^k)\big|^{\theta} + \big|f(y^{k})-\bar f\big|^{\theta}\\
&\le \big(\ell\|w^k-y^k\|\big)^{\theta}+ C\left[2\big(\ell+1\big)\e_k^{1/2} + \ell\big\|u^{k+1}-w^k\big\|\right] \\
&\le \ell \e_k^{1/2}+  C\left[2\big(\ell+1\big)\e_k^{1/2} + \ell\big\|u^{k+1}-w^k\big\|\right] \\
&\le \widehat{C} \left[f(u^k)-f(u^{k+1})\right]^{1/2},
\end{split}
\end{equation}
where the third inequality is due to $0^{\theta}=0$ and the concavity of the function $t\mapsto t^{\theta}$ as $t\ge 0$. Observe that $\widehat{C}$ depends only on $\sigma_g, \sigma_h, \ell$, and $C$. Letting $\rho_k:=f(u^k)-\bar f$ as $k\in \N$, it follows from the relationships in \eqref{eq:beforerho} that
\begin{equation*}
\rho_{k+1}^{2\theta} \le \widehat{C}^2 \big(\rho_k - \rho_{k+1}\big)\;\text{ for all large } k\in \N.
\end{equation*}

{\bf Case~2:} {\em Assumption {\rm(A3)} holds}. In this setting, a similar line of reasoning as presented in Case~1 can be applied. However, 
the role of $y^k \in \B(w^k,\e^{1/2\theta}_k\big)$ in \eqref{eq:getbk} and \eqref{eq:Fyk} is now replaced by selecting $x^{k+1}\in\B\big(u^{k+1},\e_k^{1/2\theta}\big)$. 
In both cases, we eventually get the estimate 
\begin{equation}\label{eq:alge-seq}
\rho^{2\theta}_{k+1}\le \widetilde{C}(\rho_k - \rho_{k+1})\quad\text{ for all large }\;k\in \N
\end{equation}
with some $\widetilde{C}>0$. Applying finally the standard device (see, e.g., \cite[Lemma~4.8]{AFV18} and the references therein) allows us to deduce from \eqref{eq:alge-seq} the convergence rates claimed in assertions (i) and (ii) and thus to complete the proof of the theorem.
\end{proof} 

Observe that when $\mathcal{H}$ is a finite-dimensional Euclidean space, the requirement that
\begin{equation}\label{eq:inA4}
\text{``every subsequence of $\{u^k\}$ has a convergent subsequence''},
\end{equation}
imposed in (A4), is clearly equivalent to the boundedness of the iterates $\{u^k\}$. The latter assumption is frequently used in the studies examining convergence rates of DCA in finite dimensions; see \cite[Theorem~2]{ijcai} for the exact case and \cite[Theorem~5]{PHLH22} in the presence of some inexactness. Below we discuss the availability of \eqref{eq:inA4} in the general framework of infinite-dimensional Hilbert spaces.

\begin{Remark}\rm 
If there exists a Banach space $X$ with {\em compact embedding} $X\subset\subset \mathcal{H}$ (see, e.g., \cite[page~286,~Section~5.7]{Evans}) and $\{u^k\}$ is bounded in $X$, then \eqref{eq:inA4} is fulfilled. Indeed, any subsequence $\{u^{k_n}\}$ of $\{u^k\}$ is bounded in $X$. Then observe by the definition of compact embedding $X\subset \subset \mathcal{H}$ that $\{u^{k_n}\}$ has a subsequence that strongly converges to some limit in $\mathcal{H}$. The condition \eqref{eq:inA4} is thus verified.

Consequently, a natural way to verify \eqref{eq:inA4} is to establish an a priori bound for the iterates 
$\{u^k\}$ in a more regular space 
$X$ (for instance, 
the (Sobolev) $H^1$ or (bounded variation) BV-type space)  although the convergence analysis is formulated in the Hilbert space $\mathcal{H}$ (e.g., $L^2$), where the iterates are generated. In such situations, compactness results such as Rellich–Kondrachov \cite[Section~5.7]{Evans} ensure the compact embedding $X\subset\subset \mathcal{H}$ and hence yield \eqref{eq:inA4}  automatically. In particular, \eqref{eq:inA4} becomes a verifiable consequence of a standard {\em boundedness} estimate in $X$, rather than an abstract sequential {\em compactness} assumption on 
$\{u^k\}$ itself.
\end{Remark}

\begin{Example}\rm 
We aim to construct an example of a DC problem \eqref{eq:DC} on an {\em infinite-dimensional} Hilbert space $\mathcal{H}$, and a sequence of iterates $\{u^k\}$ generated from Algorithm~\ref{algo:i-adca} that satisfy all the assumptions of Theorem~\ref{theo:adca-theo2} including (H1)--(H2) and (A1)--(A4).

Let $\mathcal{H}:=L^{2}(0,1)$ be the Lebesgue space with the usual inner product
\[
\langle u,v\rangle := \int_{0}^{1} u(x)v(x)\,dx, 
\qquad \|u\|:=\sqrt{\langle u,u\rangle}.
\]
Define the function $e\in \mathcal{H}$ by
\[
e(x)\equiv 1 \quad \text{ for all }\quad x\in(0,1).
\]
Then $\|e\|^{2}=\int_{0}^{1}1\,dx=1$. Fix $\alpha\in(0,1)$ and define convex functions $g,h:\mathcal{H}\to\R$ by
\[
g(u):=\frac12\|u\|^{2}, 
\qquad
h(u):=\frac{\alpha}{2}\,\langle u,e\rangle^{2} \quad (u\in \mathcal{H}).
\]
In particular, $h$ is of class $\mathcal{C}^{1,1}$ on $\mathcal{H}$ and hence satisfies (H1).

The objective function
\[
f(u):=g(u)-h(u)
=\frac12\|u\|^{2}-\frac{\alpha}{2}\langle u,e\rangle^{2}\quad (u\in \mathcal{H})
\]
is of the DC type on $\mathcal{H}$. Observe that
\begin{equation}\label{eq:fu>0}
f(u) \ge \dfrac{1}{2}\|u\|^2 - \dfrac{\alpha}{2}\|u\|^2 = \left(\dfrac{1-\alpha}{2}\right)\|u\|^2\ge 0 = f(0)\;\mbox{ for all }\;u\in \mathcal{H},
\end{equation}
and therefore condition (H2) is satisfied.

Next we verify that $f$ possesses the Polyak-\L ojasiewicz property. Since
$$
\nabla f(u) = u - \alpha\, \la u,e\ra\, e,\quad \nabla f(0) =0\;\mbox{ as }\;u\in{\cal H},
$$
we readily have the estimate
$$
\begin{aligned}
\la \nabla f(u) -\nabla f(v),u-v\ra &= \big\la u-v -\alpha \,\la u-v,e\ra\, e\, , \, u-v\big\ra\\
&= \|u-v\|^2 -\alpha\la u-v,e\ra^2\\
&\ge (1-\alpha)\|u-v\|^2\;\mbox{ for all }\;u,v\in \mathcal{H}
\end{aligned}
$$ 
which comes from $|\la u-v,e\ra| \le \|u-v\|$. In particular, $\nabla f(u)=0$ yields $u=0$, so the $M$-stationary point is unique and equals $\overline{u}:=0$. Moreover, it follows from \cite[Theorems~5.40 and 6.10]{nam} that $f$ is strongly convex with modulus $1-\alpha$. Consequently, we arrive at the inequality 
\begin{equation}\label{eq:s-convex}
f(v)\ge f(u) + \la \nabla f(u),v-u\ra + \left(\dfrac{1-\alpha}{2}\right)\|v-u\|^2,\quad u,v\in\mathcal{H}.
\end{equation}
At $v=\overline{u}=0$, \eqref{eq:fu>0} and \eqref{eq:s-convex} 
imply that
\begin{equation*}
\begin{aligned}
0\le f(u)-f(\overline{u}) = f(u) - f(0) &\le \la \nabla f(u), u\ra -\left(\dfrac{1-\alpha}{2}
\right)\|u\|^2\\
&= -\left(\dfrac{1-\alpha}{2}\right)\left\|u-\dfrac{\nabla f(u)}{1-\alpha}\right\|^2 + \dfrac{\big\|\nabla f(u)\big\|^2}{2(1-\alpha)}\\
&\le \dfrac{\big\|\nabla f(u)\big\|^2}{2(1-\alpha)},
\end{aligned}
\end{equation*}
which brings us to the estimate
\begin{equation*}
|f(u)-f(\overline{u})|^{1/2} \le \dfrac{1}{\sqrt{2(1-\alpha)}}\|\nabla f(u)\|, \quad u\in \mathcal{H}.
\end{equation*}
Therefore, the Polyak-\L ojasiewicz property of $f$ (and thus assumption (A1)) holds with $C:=1/\sqrt{2(1-\alpha)}$, $\varepsilon>0$ arbitrary, and $\theta:=1/2$. Furthermore, we deduce from
\[
\nabla g(u)=u,
\qquad
\nabla h(u)=\alpha\,\langle u,e\rangle\,e
\]
that both $g$ and $h$ are $\mathcal{C}^{1,1}$ on $\mathcal{H}$ with Lipschitz moduli $1$ and $\alpha$, respectively. This tells us that the assumptions in (A2) and (A3) are fulfilled.

Finally, we specify a sequence $\{u^k\}$ generated by Algorithm~\ref{algo:i-adca} such that (A4) holds for it. Let us start with $u^0:= e$. At each iteration, choose in Step~1 $w^k:=u^k$. In Step~2, take
\[
v^k:=\nabla h(w^k)=\alpha\,\langle u^k,e\rangle\,e.
\]
In Step~3, compute $u^{k+1}$ as the (exact) minimizer of
\[
g(u)-\langle v^k,u\rangle
=\frac12\|u\|^{2}-\langle v^k,u\rangle = \dfrac{1}{2}\big\|u-v^k\big\|^2 - \dfrac{1}{2}\big\|v^k\big\|^2,
\]
which clearly ensures that
\[
u^{k+1}=v^k=\alpha\,\langle u^k,e\rangle\,e,\quad k\ge 0.
\]
Proceeding by induction leads us to
\[
u^k=\alpha^{k}\,\langle u^0,e\rangle\,e = \alpha^{k}e,
\quad\text{and hence}\quad
\|u^k\|=\alpha^{k}\to 0 \text{ as }k\to\infty.
\]
In particular, $\{u^k\}$ converges strongly to $0$ in $L^{2}(0,1)$, which verifies (A4).
\end{Example}

\section{Optimal Control of Elliptic Equations with \texorpdfstring{$L^{1-2}$}{L1-2} Sparsity-Enhanced Regularization via Inexact Adaptive DCA}\label{sec:appl-PDECO}

In this section, we use our main Algorithm~\ref{algo:i-adca} for infinite-dimensional DC optimization to design a constructive procedure for solving an optimal control problem governed by elliptic equations with sparsity regularization. The needed preliminaries on functional analysis, measure theory, and partial differential equations can be found in the standard references; see, e.g., \cite{Evans,Troltzsch}.

Consider a nonempty open bounded set $\Omega\subset \R^n$ $(n=2,3)$ whose boundary $\Gamma$ is Lipschitz continuous and whose Lebesgue measure is denoted by $|\Omega|$. This framework clearly covers the case when $\Gamma$ is {\em polygonal}, i.e., $\Gamma$ consists of a finite number of straight-line segments.

We are interested in the following {\em elliptic optimal control} problem
associated with an {\em $L^{1-2}$ regularization} and {\em hard control constraints} (see, e.g., \cite{Ulbrich09,Stadler09,Troltzsch}):
\begin{eqnarray}
\underset{y \in H_0^1(\Omega),\, u \in L^2(\Omega)}{\operatorname{minimize}} && \frac{1}{2} \|y-y_d\|_{L^2(\Omega)}^2+\frac{\alpha}{2} \|u-u_d\|^2_{L^2(\Omega)} +\beta \big( \|u\|_{L^1(\Omega)}-\|u\|_{L^2(\Omega)}\big) \notag \\
\text {subject to } && A y=\phi+u \quad \text { in } \Omega, \label{eq:mainprob}\\
&& \beta_1 \leq u \leq \beta_2 \quad \text { a.e. in } \Omega .\notag
\end{eqnarray}
When $\beta=0$, the model corresponds to the well-known \emph{optimal heat-source} problem (as in Tr\"oltzsch \cite[Section~1.2]{Troltzsch}) with $u(\cdot)$ modeling a heat source in $\Omega$ generated by, e.g., electromagnetic induction or microwaves. Stadler \cite{Stadler09} encourages sparsity by augmenting the objective with the nonsmooth penalty $\beta\|u\|_{L^1 (\Omega)}$. The $L^{1-2}$-regularized model, treated in \cite{ZSYD24} by using a completely different inexact DCA variant, is motivated by the computational studies in Yin et al. \cite{Lou}, which show that $\ell^{1-2}:=\|\cdot\|_{1}-\|\cdot\|_2$ persistently outperforms $\ell^1$ in highly coherent settings (including image denoising) while yielding stronger sparsity and better recovery. 

Note that the homogeneous Dirichlet boundary condition $y=0$ on $\Gamma$ is expressed by $y\in H_0^1 (\Omega)$. We suppose that $y_d \in L^2(\Omega)$, $u_d \in L^{\infty}(\Omega)$, $\phi \in L^3(\Omega)\subset H^{-1}(\Omega)$, and the pointwise control bounds $\beta_1<0,\beta_2>0$ are finite. The term with $\alpha>0$ representing the cost of control, produces a smoother solution to problem \eqref{eq:mainprob}. For this reason, $\alpha$ is called a {\em regularization parameter}. The {\em sparsity-enhanced parameter} $\beta>0$ leads to an optimal control with a sparse structure so that its support is small. In simple terms, as $\beta$ increases, the support of the optimal control becomes smaller. In the state equation, $A\in \mathscr{L}\big(H_0^1 (\Omega),H^{-1}(\Omega)\big)$ denotes the {\em elliptic} linear bounded operator from $H_0^1 (\Omega)$ to $H^{-1}(\Omega)$, the continuous dual of $H_0^1(\Omega)$ defined by 
\begin{equation}\label{eq:A}
Ay:= -\sum_{i,j=1}^n \big[a_{ij}(x)y_{x_i}\big]_{x_j} + a_0 (x) y\quad \text{ for all } y\in H_0^1 (\Omega),
\end{equation}
where $y_{x_i}$ denotes the $i^{\mathrm{th}}$-partial derivative of $y$, the coefficients $\{a_{ij}\}_{i,j=1}^n$ are symmetric and Lipschitz continuous on the closure of $\Omega$, and where $a_0\in L^{\infty}(\Omega)$ is such that $a_0\ge 0$ a.e. in $\Omega$. Furthermore, $A$ satisfies the {\em $H_0^1$-elliptic property} meaning that there exists $\theta>0$ such that
$$
\sum_{i,j=1}^n a_{ij}(x) \xi_i \xi_j \ge \theta \|\xi\|^2 \ \text{ for a.e. } x\in \Omega \text{ and for all } \xi \in \R^n.
$$

We first observe in the following remark that under our standing assumptions, the state equation in \eqref{eq:mainprob} possesses a unique solution $y=y(u)$. 

\begin{Remark}\rm 
\quad 
\begin{enumerate}
\item[\bf(i)] For every $u \in L^2 (\Omega)$, the elliptic state equation $Ay=\phi+u$ possesses a {\em unique solution} $y=y(u) \in H_0^1 (\Omega)$. This is a direct consequence of \cite[Lemmas~9.2~and~9.3]{Ulbrich11}. For a more comprehensive result applicable to broader settings of the state equation, see Casas et al. \cite[Theorem~1.1.2]{HandbookNA}.

\item[\bf(ii)] From (i), the relation $u\mapsto y(u)-y_d$ is an {\em affine mapping} from $L^2 (\Omega)$ to $H_0^1 (\Omega)$, i.e., on $L^2 (\Omega)$ we have $y(u)=L(u)+w_0$  for some linear mapping $L:L^2 (\Omega)\to H_0^1 (\Omega)$ and for some $w_0\in H_0^1 (\Omega)$. In particular, the function $u\mapsto \frac{1}{2} \|y(u)-y_d\|_{L^2(\Omega)}^2$ is {\em convex}.
\end{enumerate}
\end{Remark}

Now let us introduce the equivalent {\em reduced problem} associated with \eqref{eq:mainprob}:
\begin{equation}
\begin{split}
\underset{u \in L^2(\Omega)}{\operatorname{minimize}} \quad & \frac{1}{2} \big\|y(u)-y_d\big\|_{L^2(\Omega)}^2+\frac{\alpha}{2} \|u-u_d\|^2_{L^2(\Omega)} + \beta \big( \|u\|_{L^1(\Omega)}-\|u\|_{L^2(\Omega)}\big) \\
\text {subject to}\quad
& \beta_1 \leq u \leq \beta_2 \quad \text { a.e. in } \Omega.
\end{split}\label{reduced1}
\end{equation}
We set $U:=L^2(\Omega)$, where $L^2(\Omega)$ is identified with its dual by the Riesz representation theorem, and set $Y:=H_0^1(\Omega)$. Denote the set of {\em admissible controls} as
\begin{equation}\label{eq:admissible}
U_{{\rm ad}}:=\{u\in U \mid \beta_1 \le u \le \beta_2 \ \text { a.e. in }\ \Omega\}.
\end{equation}
Since the strong convergence in $L^2 (\Omega)$ yields almost everywhere convergence up to a subsequence, $U_{{\rm ad}}$ is a nonempty closed convex subset of $U$. Define the functions $g,h\in \Gamma (U)$ by
\begin{eqnarray}
&& g(u):=\frac{1}{2} \|y(u)-y_d\|_{U}^2+\frac{\alpha}{2} \|u-u_d\|^2_{U} + \beta \|u\|_{L^1(\Omega)} + \delta_{{U}_{{\rm ad}}}(u)\quad \text{ for all }\;u\in U,\label{eq:G}\\
&& h(u):=\beta \|u\|_{U} \quad \text{ for all }\;u\in U \label{eq:H}
\end{eqnarray}
via the {\em indicator function} $\delta_{{U}_{{\rm ad}}}$ of ${U}_{{\rm ad}}$ given by $\delta_{{U}_{{\rm ad}}}(u)=0$ when $u\in {U}_{{\rm ad}}$, and $\infty$ otherwise. Then we represent the reduced problem \eqref{reduced1} as a {\em DC optimization problem} of the form
\begin{equation}\label{eq:mainprob-DCA}
\underset{u\in U}{\operatorname{minimize}}\;\;\;  f(u):=g(u)-h(u),
\end{equation}
in which $g$ is strongly convex with $\sigma_g=\alpha$ while $\sigma_h=0$.

To be able to apply our inexact adaptive DC algorithm to problem \eqref{eq:mainprob-DCA} with the above data, observe first that the function $h(u)=\beta \|u\|_{U}$ is differentiable with the gradient $\nabla h(u)=
\beta u/\|u\|_{U}$ at each $u\neq 0$ and that $0$ is a subgradient of $h$ at $u=0$. Regarding Step~2 of Algorithm~\ref{algo:i-adca}, we choose in each iteration the vector
\begin{equation*}
v^k = \begin{cases}
0 &\text{if }w^k=0,\\
\beta\dfrac{w^k}{\|w^k\|_U} &\text{if }w^k\neq 0.
\end{cases}
\end{equation*}
This results in solving the subproblem
\begin{equation}\label{subproblem}
u^{k+1} \in
\begin{cases}
\e_k\text{-}\mathrm{argmin}_{u\in U} g(u) &\text{if }w^k=0,\\
\e_k\text{-}\mathrm{argmin}_{u\in U} \Bigg[g(u)-\Bigg\la \beta\dfrac{w^k}{\|w^k\|_U},u\Bigg\ra_{U}\Bigg] &\text{if }w^k \neq 0,
\end{cases}
\end{equation}
and leads us to the following elaboration of Algorithm~\ref{algo:i-adca} to
solve \eqref{eq:mainprob-DCA}.

\begin{algorithm}[H]
\caption{I-ADCA-PDECO: Inexact Adaptive DCA for PDE Control}
\label{algo:i-adca-pdeco}
\textbf{Initialization:} $u^0\in {U}_{{\rm ad}}$,  
reduction factor $\gamma\in (0,1),\;\varepsilon_0\in(0,1]$. \\
Set $k=0$.
\begin{algorithmic}[1]
\STATE Choose $w^k$ satisfying $f(w^k) \le f(u^k)$.
\STATE Find $u^{k+1} \in U$ such that
\[
u^{k+1} \in
\begin{cases}
\varepsilon_k\text{-}\mathrm{argmin}_{u\in U} g(u) & \text{if } w^k = 0,\\[0.5em]
\varepsilon_k\text{-}\mathrm{argmin}_{u\in U}
\Bigg[
g(u) - \Bigg\langle \beta \dfrac{w^k}{\|w^k\|_U}, u \Bigg\rangle_{U}
\Bigg] & \text{if } w^k \neq 0.
\end{cases}\]
\STATE {\bf While} 
\[
\varepsilon_k > \frac{\sigma_g + \sigma_h}{32}\big\|u^{k+1} - w^k\big\|^2,
\]
set $\varepsilon_k \gets \gamma\,\varepsilon_k$ and restart from Step~2 with the new $\varepsilon_k$.
\STATE Choose $\varepsilon_{k+1} \in (0,\varepsilon_k]$ and set $k \gets k + 1$.
\end{algorithmic}
\noindent \textbf{Until} Stopping criteria.
\end{algorithm}

It follows from Section~\ref{sec:i-ADCA} that Algorithm~\ref{algo:i-adca-pdeco} is {\em well-defined} since the assumptions in (H1) and (H2) are fulfilled as discussed below. 

\begin{Remark} $\,$\rm  \label{prop:inffinite}
\begin{enumerate}
\item[\bf(i)] The function $h:=\beta\|\cdot\|_U$ is clearly continuous and satisfies (H1).

\item[\bf(ii)] Assumption (H2) holds for the DC problem \eqref{eq:mainprob-DCA}. Indeed, for $u\in {U}_{{\rm ad}}$ we get 
$$
\|u\|_{U}\le |\Omega|^{1/2} \cdot \max \{|\beta_1|,|\beta_2|\}<\infty.
$$
Thus it easily follows from \eqref{eq:G} and \eqref{eq:H} that 
$$
\inf_{u\in U} f(u) = \inf_{u\in U}  \big[g(u)-h(u)\big] \ge \inf_{u\in U} \big[-\beta\|u\|_U\big] > -\infty.
$$
Therefore, the sequences $\{u^k\}$, $\{w^k\}$, and $\{v^k\}$ generated by Algorithm~\ref{algo:i-adca-pdeco} from any starting point $u^0\in {U}_{{\rm ad}}$, are well-defined.

\item[\bf(iii)] Subproblem \eqref{subproblem} admits a unique solution $u^{k+1}$ when $\e_k=0$. Indeed, since the objective functions in \eqref{subproblem} are convex, continuous, and bounded from below, the proof easily follows from that of either Ulbrich \cite[Lemma~9.4]{Ulbrich11}, or Casas et al. \cite[Theorem~1.2.6]{HandbookNA}. The strongly convex term $\dfrac{\alpha}{2}\|u-u_d\|^2_U$ in $g(\cdot)$ from \eqref{eq:G} clearly yields the solution uniqueness.
\end{enumerate}
\end{Remark}

\section{Finite Element Discretization and Numerical Implementation}\label{finite-elem}

For solving subproblem \eqref{subproblem} in Algorithm \ref{algo:i-adca-pdeco} numerically, we make use of the so-called ``first discretize then optimize" approach beginning with a suitable discretization of the problem based on a {\em finite element approximation}. The ideal scenario is described as follows.

\subsection{Finite Element Scheme}\label{finite-scheme}

Suppose that $\Omega \subset \R^2$ is an open bounded convex set whose boundary is {\em polygonal}. Following Glowinski \cite[Chapter~IV,~Section~2.5]{Glowinski}, we consider a family of regular triangulations $\{\mathcal{T}^h\}_{0<h\le 1}$ of $\Omega$ with the mesh size $h>0$ possessing the properties:
\begin{itemize}
\item Each triangulation 
$\mathcal{T}^h$ is defined by
$$
\mathcal{T}^h :=\left\{T_i^h \subset \R^2 \;\Big|\; T_i^h \text{ is a closed triangle, } i=1, \ldots, t^h\right\} \text{ for some }t^h \in \N.
$$

\item $\bigcup_{i=1}^{t^h} T_i^h = \overline{\Omega}$, the closure of $\O$.

\item For all $i \neq j$, we have that $\operatorname{int} T_i^h \cap \operatorname{int} T_j^h=\varnothing$ and that $T_i^h \cap T_j^h$ must either be a common edge, or a common vertex, or the empty set.

\item The mesh size $h$ stands for the length of the longest edge among all triangles within the triangulation $\mathcal{T}^h$, and $h_T$ denotes the length of the longest edge in a particular triangle $T\in \mathcal{T}^h$. For each $T\in \mathcal{T}^h$, let $R_T$ be the diameter of the largest ball contained in $T$. Moving on, suppose that there exist two
constants $\rho,R>0$ satisfying
\begin{equation*}
\dfrac{h_T}{R_T}\le R,\quad \dfrac{h}{h_T}\le \rho \quad \text{ for all }\; T\in \mathcal{T}^h,\; \text{ for all }\; 0<h\le 1.
\end{equation*}
\end{itemize}

Then we define the sets
$$
\begin{aligned}
Y^h &:=\left\{y^h \in C(\overline{\Omega})\;\Big|\; y^h|_T \text { is (affine) linear for all } T \in \mathcal{T}^h \text{ and } y^h|_{\Gamma}=0\right\}\subset Y, \text{ and}\\
U^h &:= \Big\{u^h\in L^{\infty}(\Omega)\;\Big|\; u^h|_{T} \text{ is constant for all }T\in \mathcal{T}^h\Big\} \subset U,
\end{aligned} 
$$
as finite-dimensional counterparts (cf. Ciarlet \cite{Ciarlet75}) to $Y=H_0^1 (\Omega)$ and $U=L^2 (\Omega)$, respectively.\vspace*{0.05in}

Let $\Sigma^h$ be the set of all vertices in the triangulation $\mathcal{T}^h$ while $\Sigma_0^h:=\{P_1^h,\ldots,P_{v^h}^h\}$ $(v^h\in \N)$ is the collection of all interior vertices. By \cite[Lemma~2.10]{Knabner}, for each inner vertex $P_i^h$ $(i=1,\ldots,v^h)$ there exists a unique function $\psi^h_i\in Y^h$ satisfying $\psi^h_i (P_i^h)=1$ and $\psi^h_i (Q)=0$ for all $Q\in \Sigma^h\setminus\{P^h_i\}$. Obviously, the collection $\{\psi^h_i\}_{i=1}^{v^h}$ forms a basis for $Y^h$ and each $y^h \in Y^h$ can be represented as the sum
\begin{equation*}
y^h = \sum_{i=1}^{v^h} y^h (P_i^h) \psi_i^h.
\end{equation*}
On the other hand, each element $u^h \in U^h$ admits the unique representation  
\begin{equation*}
u^h = \sum_{i=1}^{t^h} u^h_i \chi_{i}^h
\end{equation*}
via the characteristic function $\chi_i^h$ of the triangle $T_i^h\in \mathcal{T}^h$ and the constant $u_i^h:=u^h|_{T_i^h}$. In this case, $\{\chi_i^h\}_{i=1}^{t^h}$ is a basis for $U^h$ satisfying \cite[Assumption~4.2]{Wachsmuth09}. Then for each $u^h \in U^h$, we define in the spirit of \eqref{eq:mainprob}-\eqref{eq:A} the {\em associated discrete state} as the unique solution $y^h (u^h)\in Y^h$ to
\begin{equation}\label{eq:dis-state}
\sum_{i,j=1}^n \int_{\Omega}a_{ij}y^h_{x_i}\varphi^h_{x_j}\, dx + \int_{\Omega}a_0 y^h \varphi^h\, dx = \int_{\Omega} (\phi+u^h)\varphi^h\, dx\quad \text{ for all } \varphi^h \in Y^h,
\end{equation}
the existence of which is a direct application of Brouwer's fixed-point theorem (see \cite[page~135]{HandbookNA}). Further, the {\em discrete control-to-state mapping} $u^h \in U^h \mapsto y^h(u^h)\in Y^h$ is of class ${\cal C}^2$; see \cite[page~212]{Ulbrich11}. The discrete counterpart set ${U}_{{\rm ad}}^h$ to ${U}_{{\rm ad}}$ in \eqref{eq:admissible} is given as 
\begin{equation*}
{U}_{{\rm ad}}^h := \big\{u^h \in U^h \;\big|\; \beta_1 \le u^h_i \le \beta_2 \, \text{ for all }\, i=1,\ldots,t^h\big\}. 
\end{equation*}

Now we replace the exact version of the subproblem in \eqref{subproblem} associated with some $w^k\in U\setminus \{0\}$ by its finite-dimensional {\em discretized form}:
\begin{equation}\label{eq:discretized}
\begin{array}{ll}
\underset{u^h \in {U}_{{\rm ad}}^h}{\operatorname{minimize}} \quad&\disp\dfrac{1}{2}\int_{\Omega} \big[y^h(u^h)-y_d\big]^2\, dx
+\disp\dfrac{\alpha}{2}\int_{\Omega}\big(u^h-u_d\big)^2\,dx\\
&+\, \beta \disp\sum_{i=1}^{t^h}|u^h_i|\cdot |T_i^h| -\disp\dfrac{\beta}{\|w^k\|_{U}}\disp\int_{\Omega} w^k u^h\, dx.
\end{array}
\end{equation}
The {\em existence} and {\em uniqueness} of the optimal solution to the discrete problem \eqref{eq:discretized} are ensured by a slight modification of the proofs in \cite[Lemma~9.3]{Ulbrich11} and \cite[Theorem~1.4.7]{HandbookNA} by taking into account that the  objective function in \eqref{eq:discretized} is continuous and strongly convex and that the discrete admissible set ${U}^h_{\text{ad}}$ is closed and convex being a weakly sequentially closed subset of $U$. The following {\em first-order necessary optimality conditions} for \eqref{eq:discretized} can be deduced from \cite[Equations~(1.4.11)--(1.4.13), (1.4.15)--(1.4.17)]{HandbookNA}, where the formula for subdifferentials of the $L^1$-norm function is utilized \cite[page~137]{HandbookNA}. 

\begin{Theorem}\label{theo:1st-order}
Let $\overline{u}^h\in U^h$ be the optimal solution to problem \eqref{eq:discretized}. Then there exist $\overline{y}^h,\overline{z}^h\in Y^h$, and $\overline{\lambda}^h\in \partial \big(\|\cdot\|_{L^1 (\Omega)}\big)(\overline{u}^h)$ such that 
\begin{align}
&\sum_{i,j=1}^n \int_{\Omega}a_{ij}\overline{y}^h_{x_i}\varphi^h_{x_j} \,dx + \int_{\Omega}a_0 \overline{y}^h \varphi^h\, dx = \int_{\Omega}\big(\phi+\overline{u}^h\big)\varphi^h\, dx\quad \text{ for all } \varphi^h \in Y^h,   \label{eq:stateeq}\\
&\sum_{i,j=1}^n \int_{\Omega}a_{ij}\varphi^h_{x_i} \overline{z}^h_{x_j} \,dx + \int_{\Omega}a_0 \varphi^h \overline{z}^h \,  dx = \int_{\Omega}\big(\overline{y}^h-y_d\big)\varphi^h \, dx\quad \text{ for all } \varphi^h \in Y^h,   \label{eq:adjointstateeq}\\
&\overline{u}^h_l = \mathrm{proj}_{[\beta_1,\beta_2]}\left[-\dfrac{1}{\alpha}\left(q^h_l-\mathrm{proj}_{[\beta_1,\beta_2]}q^h_l\right)\right]\quad \text{ for all } l=1,\ldots,t^h, \label{eq:Lu=0}
\end{align}
where each $q^h_l$ is an average integral given by
\begin{equation}\label{eq:qih}
q^h_l := \dfrac{1}{|T_l^h|}\int_{T^h_l} \left(\overline{z}^h-\alpha u_d-\dfrac{\beta}{\|w^k\|_{U}}w^k\right)dx.
\end{equation}
\end{Theorem}

Building on this, a concrete numerical implementation of the discretized problem \eqref{eq:discretized} associated with the distributed optimal control problem for elliptic equation is presented in Subsection \ref{subsec:nume} via solving the system of first-order optimality conditions 
\eqref{eq:stateeq}--\eqref{eq:Lu=0} by using the \emph{semismooth Newton method}; see, e.g., Hinterm\"uller et al. \cite{hinter02},  Stadler \cite{Stadler09}, and Ulbrich \cite{Ulbrich11}.\vspace*{0.03in}

The next step is to demonstrate the convergence of the discretized solutions $\overline{u}_h$ to the continuous one $\overline{u}$ as $h\to 0$ and
to establish \emph{error estimates/rates of convergence} for the difference between $\overline{u}_h$ and $\overline{u}$ by means of the finite element discretization under consideration.

\begin{Theorem}\label{theo:error-estimate}
Denote by $\overline{u}$ and $\overline{u}^h$, respectively, the optimal solution to the second exact subproblem in \eqref{subproblem} and to the discrete version \eqref{eq:discretized} with respect to a discretization parameter $h\in (0,1]$. Then there exists 
a constant $C>0$ independent of $h$ for which the \emph{error estimate} 
\begin{equation}\label{eq:error}
\big\|\overline{u}_h - \overline{u}\big\|_{L^{\infty}(\Omega)} \le Ch 
\end{equation}
holds for all $h$ sufficiently small. Consequently, $\big\|\overline{u}_h - \overline{u}\big\|_{L^2 (\Omega)}\le \widehat{C}h$ for some $\widehat{C}>0$.
\end{Theorem}
\begin{proof}
By \cite[Theorem~1.1.2 and Corollary~1.1.4]{HandbookNA} and the construction of \eqref{subproblem}, the second-order sufficient condition \cite[Theorem~1.3.3]{HandbookNA} is justified with $\delta:=\alpha$. Thus the result of \cite[Theorem~1.4.9]{HandbookNA} gives us a constant $C>0$ ensuring the claimed error estimate \eqref{eq:error}.
\end{proof}

\subsection{Numerical Experiments}\label{subsec:nume}

This part demonstrates numerical performance of the proposed I-ADCA-PDECO algorithm for sparse PDE-constrained optimization problem. All of our results are implemented in MATLAB R2024b running on a PC with Ryzen 5900X CPU and 32GB of memory. Our codes are built on top of the PDE-constrained optimization framework developed by Antil \cite{Antil}.

We begin by introducing the algorithmic configuration settings, which are applied to all the examples below. The discretization was implemented by using the standard piecewise linear finite element approach, i.e., 5-point stencil. In all the numerical experiments, we evaluate the accuracy of an approximate optimal solution via the corresponding residual error by choosing zeros as the initial point. The stopping criterion for Algorithm~\ref{algo:i-adca-pdeco} is given by 
\[
\frac{\|u^{k+1} - u^{k}\|_{L^2(\Omega)}}{\max\{\|u^k\|_{L^2(\Omega)}\,,1\}} \le \epsilon.
\]
We terminate the algorithm when either $\epsilon < 10^{-12}$, or the maximum number of iterations equal to 20 is reached. Consider the problem: 
\begin{eqnarray}
\underset{y \in H_0^1(\Omega),\, u \in L^2(\Omega)}{\operatorname{minimize}} && \frac{1}{2} \|y-y_d\|_{L^2(\Omega)}^2+\frac{\alpha}{2} \|u-u_d\|^2_{L^2(\Omega)} + \beta \big( \|u\|_{L^1(\Omega)}-\|u\|_{L^2(\Omega)}\big) \notag \\
\text {subject to } && A y=\phi+u \quad \text { in } \Omega, \label{eq:ex-1} \\
&& \beta_1 \leq u \leq \beta_2 \quad \text { a.e. in } \Omega \notag
\end{eqnarray}
in which $\Omega:=[0,1]\times [0,1]$ and $A:=-\Delta$ is the {\em Laplacian operator}, i.e., in \eqref{eq:A} we let $a_0\equiv 0$ and $a_{ij}\equiv \delta_{ij}$, the {\em Kronecker delta}. The discretized \emph{necessary optimality conditions} \eqref{eq:stateeq}--\eqref{eq:Lu=0} in case \eqref{eq:ex-1} ensure the existence of $\big(\overline{y}^h,\overline{z}^h,\overline{\lambda}^h\big)\in Y^h\times Y^h \times U^h$ satisfying
\begin{equation}\label{eq:1st-pdeco}
\begin{cases}
 \displaystyle\sum_{i=1}^2 \int_{\Omega}\overline{y}^h_{x_i} \varphi^h_{x_i}\,dx = \int_{\Omega} \big(\phi+\overline{u}^h\big) \varphi^h \,dx\quad \text{ for all } \varphi^h \in Y^h,\\
 \displaystyle\sum_{i=1}^2 \int_{\Omega} \varphi^h_{x_i}\overline{z}^h_{x_i} \,dx = \int_{\Omega} \big(\overline{y}^h - y_d\big)\varphi^h \,dx\quad \text{ for all } \varphi^h \in Y^h,\\
 \overline{u}^h_l = \mathrm{proj}_{[\beta_1,\beta_2]}\left[-\dfrac{1}{\alpha}\left(q^h_l-\mathrm{proj}_{[\beta_1,\beta_2]}q^h_l\right)\right]\quad \text{ for all }\;l=1,\ldots,t^h, 
\end{cases}
\end{equation}
where the sequence of average integrals $\{q^h_l\}_{l=1}^{t^h}$ is given by \eqref{eq:qih}. In what follows, the matrix $\overline{A}$ represents a discretization of the Laplacian operator (the \emph{stiffness matrix}), $M$ is the FEM Gram matrix, i.e., the \emph{mass matrix}, and the matrix $\overline{M}$ represents the discretization of the control term within the PDE constraint. Using these matrices together with $d^h:=\dfrac{w^h}{\|w^h\|_U}$,  the relationships in \eqref{eq:1st-pdeco} induce the system of equations $F(\overline{y}^h,\overline{z}^h,\overline{u}^h,\overline{\mu}^h) = 0$  defined by 
\begin{equation}\label{eq:KKT}
\begin{cases}
\overline{A}\overline{y}^h - \overline{M} \overline{u}^h - \phi^h = 0,\\
M\overline{y}^h +\overline{A}^T \overline{z}^h - My_d^h = 0,\\
\alpha M\overline{u}^h - \overline{M}^{T}\overline{z}^h+M\overline{\mu}^h -\beta d^h = 0,\\
M\Big(\overline{u}^h - \max\big\{0,\overline{u}^h +\alpha^{-1}(\overline{\mu}^h - \beta)\big\} - \min\big\{0,\overline{u}^h +\alpha^{-1}(\overline{\mu}^h + \beta)\big\}\\
\quad + \max\big\{0,(\overline{u}^h-\beta_2)+\alpha^{-1}(\overline{\mu}^h-\beta)\big\} + \min\big\{0,(\overline{u}^h-\beta_1)+\alpha^{-1}(\overline{\mu}^h+\beta)\big\}\Big) = 0.
\end{cases}
\end{equation}
The latter system of equations can be solved efficiently by using the {\em primal-dual active set method} developed in \cite[Algorithm~2]{Stadler09}. At each inner iteration, the violation of the discrete KKT system is measured by the residual vectors
$r_{\mathrm{state}}, r_{\mathrm{adj}}, r_{\mathrm{stat}}$, and $r_{\mathrm{comp}}$ associated with the left-hand sides of the four equations in \eqref{eq:KKT}. Consequently, we define the residual
\[
\eta := \max\bigl\{
\|r_{\mathrm{state}}\|_{L^2(\Omega)},\,
\|r_{\mathrm{adj}}\|_{L^2(\Omega)},\,
\|r_{\mathrm{stat}}\|_{L^2(\Omega)},\,
\|r_{\mathrm{comp}}\|_{L^2(\Omega)}
\bigr\}
\]
and terminate the solution of the subproblem when $\eta$ is smaller than a prescribed tolerance.\vspace*{0.05in}

Next we focus on the following test problems in \cite{Stadler09}. 

\begin{Example}\rm \label{ex:pde1}
The data for this example are as follows: the desired state and source term are
$$
y_d:=\sin (2\pi x)\sin (2\pi y) \exp (2x)/6 \quad \text{and} \quad \phi:=0,
$$
the regularization parameter is $\alpha:=10^{-4}$, and the bounds of admissible controls are $\beta_1:=-20$ and $\beta_2:=20$. We study the convergence rate of Algorithm~\ref{algo:i-adca-pdeco} when $\beta = 0.2$ and the effect that the parameter $\beta$ has on the control solution by setting it to $\beta = 0.5,0.2,0.1,0$ of $\beta_c$, a parameter that is specified later. The results are shown in Figure~\ref{fig:pde12} and Figure \ref{fig:pde13}.
\end{Example}

\begin{Example}\rm \label{ex:pde2}
In this example, the desired state and source term are
\[
y_d := \sin (4\pi x)\cos(8\pi x)\exp(2x)\quad \text{and} \quad \phi := 10\cos(8\pi x)\cos(8\pi y),
\]
respectively. We let $\alpha := 0.0001$, $\beta := 0.01 \beta_c$ and set the pointwise control bounds as $\beta_1 := -40$ and $\beta_2 := 40$.
The approximate control $u^h$ and the approximate state $y^h$ output from Algorithm~\ref{algo:i-adca-pdeco} are shown in Figure~\ref{fig:pde2}.
\end{Example}

Table~\ref{tab:example1} and Table~\ref{tab:example2} report the error between the approximate state and the discretized optimal state given by $E_2:=\|y^h-\overline{y}^h\|_{L^2(\Omega)}$. Moreover, we compare the convergence profile of I-DCA-PDECO (obtained from Algorithm~\ref{algo:i-adca-pdeco} in which $w^k:=u^k$ and Step~3 is omitted) and I-ADCA-PDECO (Algorithm~\ref{algo:i-adca-pdeco}) for Examples~\ref{ex:pde1} and \ref{ex:pde2}, respectively. The mesh here is refined from
\(h = 1/2^4\) to \(h = 1/2^7\), where we use the same parameter $\alpha = 0.01$ for Example~\ref{ex:pde1} and $\alpha = 0.0001$ for Example~\ref{ex:pde2} in all the settings. In both tables, the parameter $\beta$ was chosen as $\beta= 0.1 \beta_c$, where $\beta_c := \|\overline{A}^{-T}(M^{T}(y^{h}_d - \overline{A}^{-1}\phi^h))\|_{L^\infty(\Omega)}$.
\begin{table}[ht]
\centering
  
\resizebox{\textwidth}{!}{
\begin{tabular}{
l  
l   
l  
l   
l   
l  
l   
l    
}
\toprule
{$h$} & {\#dofs} & {$E_2$} & {Residual $\eta$}
& Metric & {DCA} & {I-ADCA}\\
\midrule
\multirow[t]{2}{*}{$1/2^{4}$}
& 225   & 2.45e-1 & 1.40e-04 & CPU time (s) & 2.5     & 1.453 \\[1pt]
&&  &   & \#iter& 5 & 9 \\
\addlinespace[2pt]

\multirow[t]{2}{*}{$1/2^{5}$}
& 961   & 2.48e-1 & 2.92e-05 & CPU time (s) & 9.922  & 5.234  \\[1pt]
&&  &   & \#iter& 5      & 7  \\
\addlinespace[2pt]

\multirow[t]{2}{*}{$1/2^{6}$}
& 3969  & 2.49e-1 & 4.75e-06 & CPU time (s) & 58.703 & 30.766 \\[1pt]
&&  &   & \#iter& 4 & 7 \\
\addlinespace[2pt]

\multirow[t]{2}{*}{$1/2^{7}$}
& 16129 & 2.49e-1 & 9.14e-07 & CPU time (s) & 305.891 & 168.844 \\[1pt]
&&  &   & \#iter& 4      & 7 \\
\bottomrule
\end{tabular}
}
\vspace{3pt}
\caption{Convergence profile comparison of I-DCA-PDECO and I-ADCA-PDECO for Example \ref{ex:pde1}}
\label{tab:example1}
\end{table}

\begin{table}[h]
\centering
\resizebox{\textwidth}{!}{
\begin{tabular}{
l   
l   
l   
l   
l   
l   
l   
l   
}
\toprule
{$h$} & {\#dofs} & {$E_2$} & {Residual $\eta$}
& Metric & {DCA} & {I-ADCA} \\
\midrule
\multirow[t]{2}{*}{$1/2^{4}$}
& 225   & 1.41e+00 & 5.92e-05 & CPU time (s) & 1.8     & 1.047\\[1pt]
&&  &   & \#iter& 4& 7 \\
\addlinespace[2pt]

\multirow[t]{2}{*}{$1/2^{5}$}
& 961   & 1.69e+00 & 3.39e-05 & CPU time (s) & 9.281  & 2.906 \\[1pt]
&&  &   & \#iter& 4& 5 \\
\addlinespace[2pt]

\multirow[t]{2}{*}{$1/2^{6}$}
& 3969  & 1.77e+00 & 6.50e-06 & CPU time (s) & 106.797 & 51.859 \\[1pt]
&&  &   & \#iter& 4& 7 \\
\addlinespace[2pt]

\multirow[t]{2}{*}{$1/2^{7}$}
& 16129 & 1.79e+00 & 1.29e-06 & CPU time (s) & 347.672 & 155.922 \\[1pt]
&&  &   & \#iter& 4 & 6 \\
\bottomrule
\end{tabular}%
}
\vspace{3pt}
\caption{Convergence profile comparison of I-DCA-PDECO and I-ADCA-PDECO for Example \ref{ex:pde2}}
\label{tab:example2}
\end{table}

Regarding the performance of all the methods in
both examples, the \(L^2\)-error \(E_2\) stays essentially constant around
\(2.5\times 10^{-1}\) in Table~\ref{tab:example1} and between \(1.4\) and \(1.8\)
in Table~\ref{tab:example2}, while the residual \(\eta\), the maximum of the norms of the four left-hand sides in the KKT optimality system \eqref{eq:KKT}, decreases by
approximately one to two orders of magnitude as the discretization is refined, reflecting the increased accuracy of the discrete solutions. At the same time,
the number of outer DCA iterations remains very small and almost unaffected by the number of degrees of freedom (between 4 and 9 in Table~\ref{tab:example1} and between 4 and 7 in Table~\ref{tab:example2}), indicating a mesh-independent iteration complexity. In terms of efficiency, I-ADCA-PDECO consistently reduces CPU time by roughly a
factor of two compared with the basic I-DCA-PDECO in both examples.

\begin{figure}[ht]
\centering
\begin{subfigure}[t]{0.4\textwidth}
\centering
\includegraphics[trim={0 0 0 0.7cm},clip,width=\linewidth]{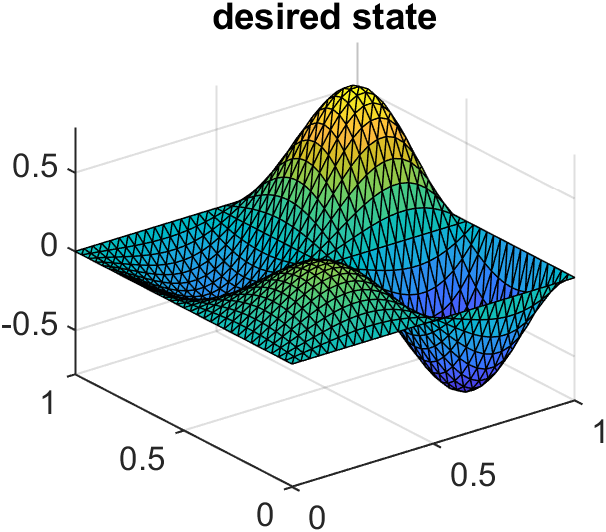}
\caption{Desired state}
\end{subfigure}
\begin{subfigure}[t]{0.4\textwidth}
\centering
\includegraphics[trim={0 0 0 0.7cm},clip,width=\linewidth]{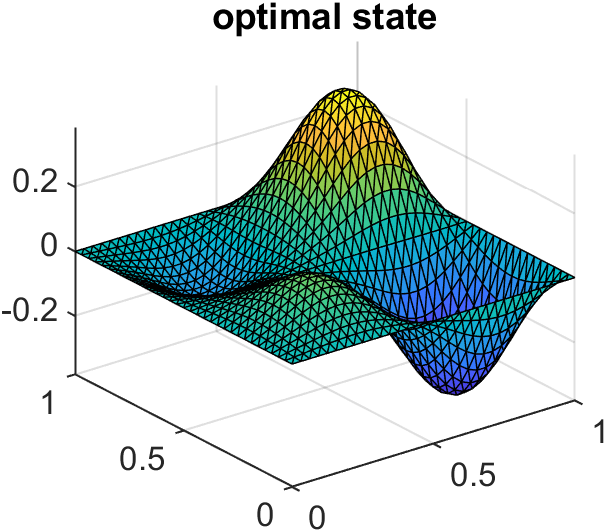}
\caption{Optimal state}
\end{subfigure}
\label{fig:pde11}
\caption{Desired state and optimal state output of I-ADCA-PDECO for Example \ref{ex:pde1}}
\end{figure}

\begin{figure}[ht]
\centering
\begin{subfigure}[t]{0.4\textwidth}
\centering
\includegraphics[trim={0 0 0 0.7cm},clip,width=\linewidth]{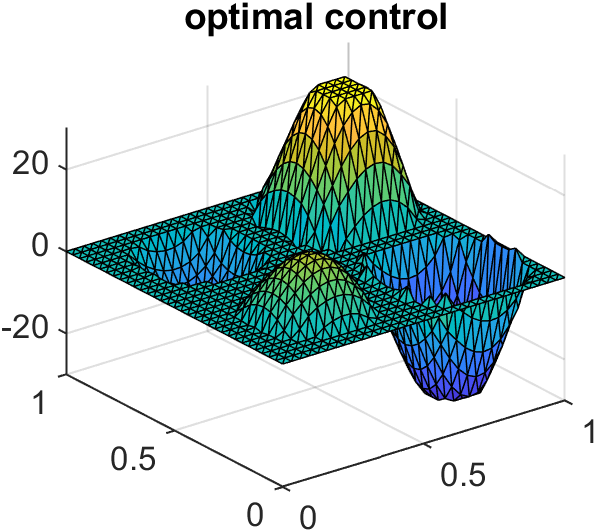}
\caption{Optimal control}
\end{subfigure}
\begin{subfigure}[t]{0.4\textwidth}
\centering
\includegraphics[trim={0 0 0 0cm},clip,width=\linewidth]{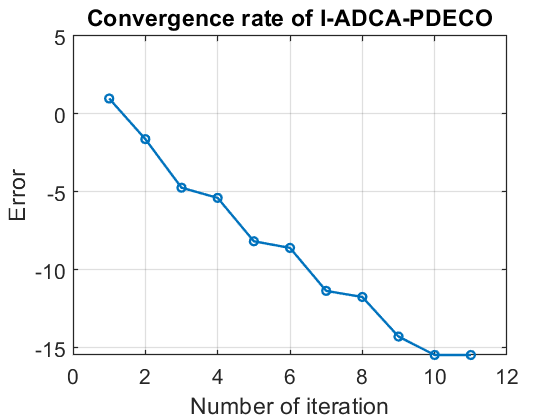}
\caption{Convergence rate of I-ADCA-PDECO}
\end{subfigure}
\caption{Optimal control output and convergence rate of I-ADCA-PDECO for Example \ref{ex:pde1}}
\label{fig:pde12}
\end{figure}

In Figure~\ref{fig:pde13}, we illustrate the influence of the sparsity–enhancing parameter $\beta$ on the structure of the optimal control in Example~\ref{ex:pde1}. As predicted by the $L^{1-2}$ regularization term, larger values of $\beta$ amplify the sparsity–inducing effect, thereby reducing the support of the control. Specifically, for higher choices of 
$\beta$, the control is activated only on small, localized regions of the domain, while being driven to zero elsewhere. In contrast, when $\beta$ is chosen smaller, the sparsity–inducing effect diminishes and the optimal control exhibits a broader spatial distribution, with nonzero values spread over a larger portion of the domain. The progression across the subfigures in 
Figure~\ref{fig:pde13} clearly demonstrates this trend, with the control profile evolving from highly concentrated and sparse patterns to more diffuse and dense structures as $\beta$ decreases.

\begin{figure}[ht]
\centering
\begin{subfigure}[t]{0.4\linewidth}
\centering
\includegraphics[trim={0 0 0 0.7cm},clip,width=\linewidth]{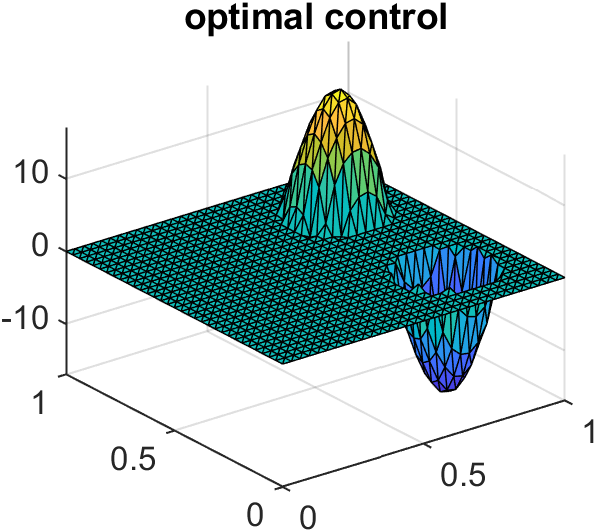}
\caption{$\beta = 0.5 \beta_c$}
\end{subfigure} 
\begin{subfigure}[t]{0.4\linewidth}
\centering
\includegraphics[trim={0 0 0 0.7cm},clip,width=\linewidth]{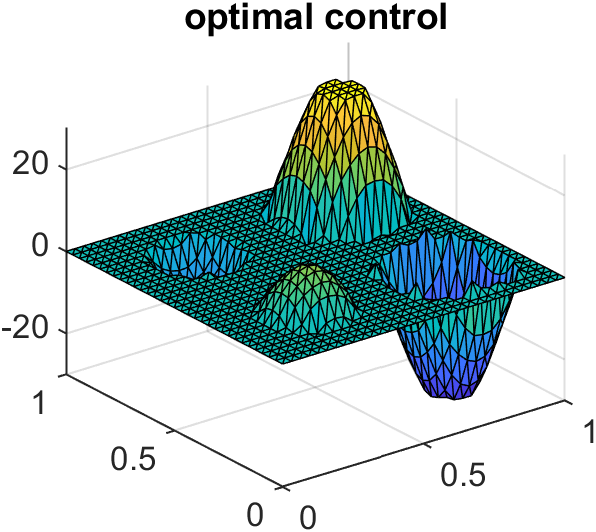}
\caption{$\beta = 0.2 \beta_c$}
\end{subfigure}

\centering
\begin{subfigure}[t]{0.4\linewidth}
\centering
\includegraphics[trim={0 0 0 0.7cm},clip,width=\linewidth]{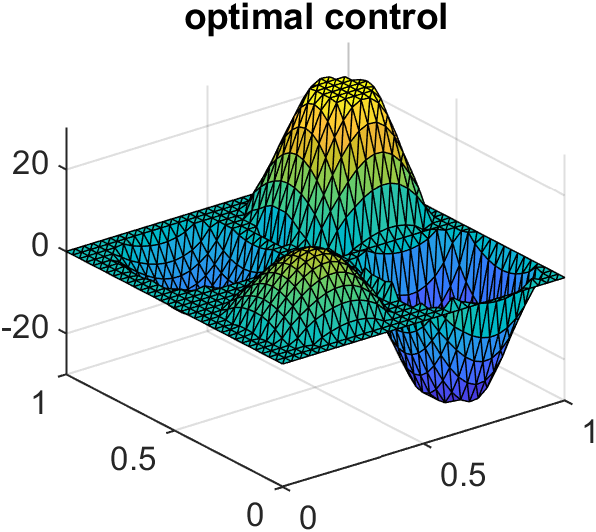}
\caption{$\beta = 0.1\beta_c$}
\end{subfigure}
\begin{subfigure}[t]{0.4\linewidth}
\centering
\includegraphics[trim={0 0 0 0.7cm},clip,width=\linewidth]{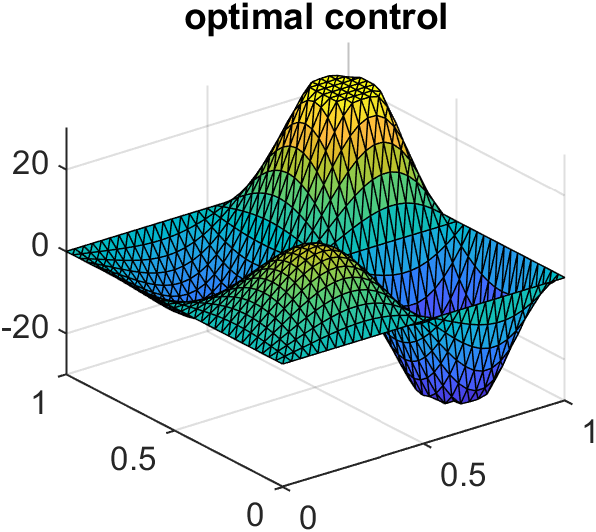}
\caption{$\beta = 0$}
\end{subfigure}
\caption{Optimal control output for Example \ref{ex:pde1} for various values of $\beta$}
\label{fig:pde13}
\end{figure}

\begin{figure}[!ht]
\centering
\begin{subfigure}[t]{0.32\linewidth}
\centering
\includegraphics[trim={0 0 0 0.7cm},clip,,width=\linewidth]{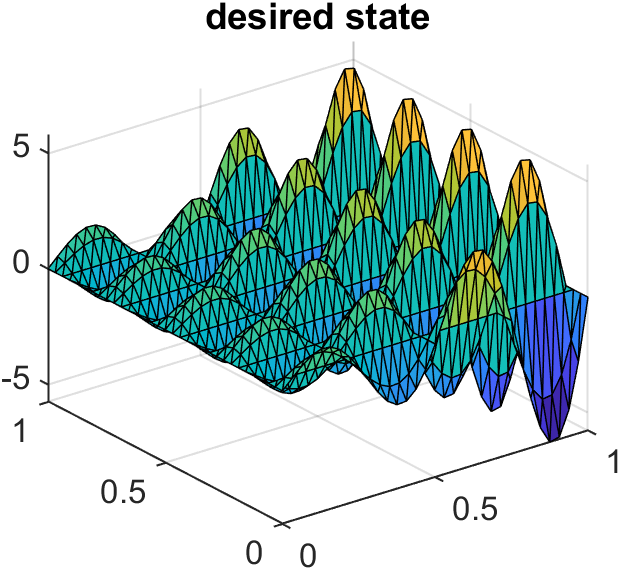}
\caption{Desired state}
\label{fig:one}
\end{subfigure}
\hfill
\begin{subfigure}[t]{0.32\linewidth}
\centering
\includegraphics[trim={0 0 0 0.7cm},clip,width=\linewidth]{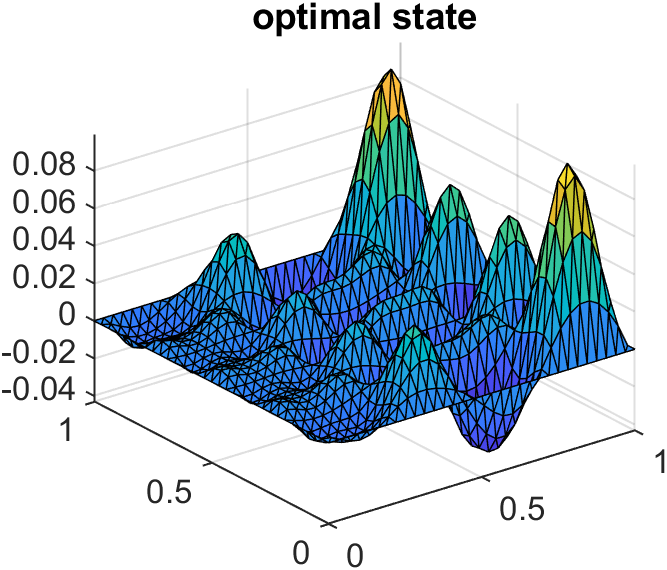}
\caption{Optimal state}
\label{fig:tw}
    
\end{subfigure}
\hfill
\begin{subfigure}[t]{0.32\linewidth}
\centering
\includegraphics[trim={0 0 0 0.7cm},clip,width=\linewidth]{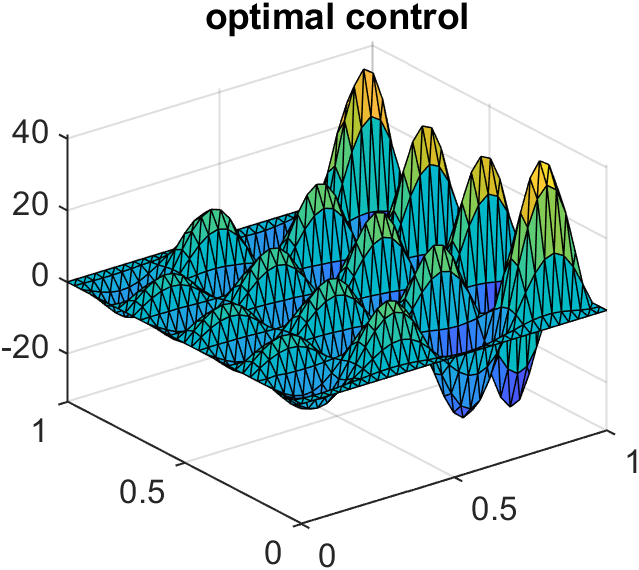}
\caption{Optimal control}
\label{fig:three}
\end{subfigure}
\caption{Desired state, optimal state, and optimal control output of I-ADCA-PDECO for Example~\ref{ex:pde2}}
\label{fig:pde2}
\end{figure}

\section{Concluding Remarks and Future Research}\label{conc}

In this work, we designed and justified an inexact adaptive DC algorithm in Hilbert spaces and applied it to sparse elliptic optimal control problems with an $L^{1-2}$-type regularization. The proposed scheme called I-ADCA (Algorithm~\ref{algo:i-adca}) allows both inexact subgradients and inexact solutions of the convex subproblems while preserving global convergence to stationary points and, under a Polyak-\L ojasiewicz-type property and mild regularity assumptions, guaranteeing explicit convergence rates. Specializing the designed algorithm to PDE-constrained optimization led us to develop the I-ADCA-PDECO method (Algorithm~\ref{algo:i-adca-pdeco}), for which we established well-posedness, uniqueness of the subproblems, and finite element error estimates. The conducted numerical experiments illustrated the ability of I-ADCA-PDECO to compute sparse controls with good performance in a benchmark distributed control problem.

Among major directions of our future research, we mention further investigations of the convergence rates for the inexact adaptive DC algorithm I-ADCA under appropriate versions of the Polyak-\L ojasiewicz property, especially in the case of lower exponents $\theta\in(0,1/2)$. We also plan to develop our approach to design and justify an appropriate inexact adaptive counterpart of the {\em boosted DCA} (BDCA) in Hilbert spaces; cf. \cite{AFV18,bmmn25} for exact versions in finite dimensions. Our future interests in this direction include the design and complexity analysis of acceleration versions of the inexact infinite-dimensional BDCA with both nonsmooth components, which exhibit significant new features in comparison with partially smooth counterparts; see \cite{fmss26} for the plain framework in finite dimensions. The proposed algorithmic developments will be tested in applications to new classes of regularized problems of PDE-constrained optimization to promote their sparsity and reduce computational costs.

\section*{Acknowledgments} The authors are grateful to Eduardo Casas and Vu Huu Nhu for useful discussions on optimal control of elliptic PDE systems.

\end{document}